\documentclass{amsart}
\usepackage{amsmath,amssymb,graphicx,url} %,ulem}
\usepackage{amsfonts,mathrsfs}
\usepackage{tikz}
\usetikzlibrary{arrows,backgrounds,decorations,automata}

\usepackage{placeins}

\usetikzlibrary{positioning,calc}

\newtheorem{thm}{Theorem}[section]
\newtheorem{lem}[thm]{Lemma}
\newtheorem{alg}[thm]{Algorithm}
\newtheorem{cor}[thm]{Corollary}
\newtheorem{prop}[thm]{Proposition}
\theoremstyle{definition}
\newtheorem{defn}[thm]{Definition}

\newtheorem{prob}[thm]{Problem}

\newcommand{\bi}{\begin{itemize}}
\newcommand{\ei}{\end{itemize}}
\newcommand{\be}{\begin{enumerate}}
\newcommand{\ee}{\end{enumerate}}
\newcommand{\bc}{\begin{center}}
\newcommand{\ec}{\end{center}}
\newcommand{\bt}{\begin{tabular}}
\newcommand{\et}{\end{tabular}}
\newcommand{\ba}{\begin{array}}
\newcommand{\ea}{\end{array}}
\newcommand{\ou}{\overline{u}}
\newcommand{\ov}{\overline{v}}

\newcommand{\N}{\mathbb N}

\newcommand{\e}{\epsilon}
\newcommand{\dcs}{$\mathcal {DCS}$}

\newcommand{\KKM}{KKM}
\newcommand{\MSunic}{MR2966442}
\newcommand{\papertwo}{JenSharif}
\newcommand{\BridGilman}{MR1420509}
\newcommand{\Cho}{cho2006introduction}
\newcommand{\BritHerm}{autostack}

\newcommand{\EEO}{MR3030521}
\newcommand{\LiptonZ}{MR0445901}

\newcommand{\LS}{MR0577064}
\newcommand{\WordProc}{MR1161694}
\newcommand{\Parry}{MR1062874}
\newcommand{\ElderBS}{MR2776987}
\newcommand{\Miller}{MR0310044}

\newcommand{\Mihailova}{MR0222179}
\newcommand{\Mitrana}{MR1605205}
\newcommand{\EKO}{MR2483126}
\newcommand{\EGaut}{MR2328170}
\newcommand{\GreibA}{MR0411257}
\newcommand{\GreibB}{MR513714}
\newcommand{\Kambites}{MR2259632}
\newcommand{\KambitesFormal}{MR2482816}
\newcommand{\Tara}{Tara}
\newcommand{\Dima}{Dima}
\newcommand{\Shapiro}{MR1313124}
\newcommand{\polypush}{MR1704452}
\newcommand{\BookGins}{MR0319412}
\newcommand{\GreibGins}{MR0300828}
\newcommand{\Minsky}{MR0140405}
\newcommand{\LakinThomas}{MR2152152}
\newcommand{\LakinThomasGCC}{MR2598993}
\newcommand{\HopUl}{MR645539}
\newcommand{\Fischer}{MR0235932}

\newcommand{\GeodProbs}{MR2747151}
\newcommand{\MiasFreeSolv}{MR2645045}
\newcommand{\HoltCoindexed}{MR2274726}
\newcommand{\HoltCocf}{MR2132375}
\newcommand{\Joerg}{MR2323454}
\newcommand{\VassThesis}{Vass}

\begin{document}
%%%%%%%%%%%%%%%%

\title{$\mathcal C$-graph  automatic groups}

\author{Murray Elder}
\address{School of Mathematical and Physical Sciences,
The University of Newcastle,
Callaghan NSW 2308, Australia}
\email{murrayelder@gmail.com}

\author{Jennifer Taback}
\address{Department of Mathematics,
Bowdoin College, Brunswick, ME 04011, USA}
\email{jtaback@bowdoin.edu}

\keywords{automatic group; Cayley graph automatic group; counter language; context-sensitive language; 
word problem; polynomial time algorithm; Baumslag-Solitar group}  
\subjclass[2010]{20F65; 	68Q45}
\date{\today}
\thanks{The first author is supported by Australian Research Council grant  FT110100178, and the second author is partially
supported by National Science Foundation grant DMS-1105407.}

\begin{abstract}
We generalize the notion of a graph automatic group introduced by Kharlampovich, Khoussainov and Miasnikov 
by
replacing the regular languages in their definition with more powerful language classes. For a fixed language class $\mathcal C$,
we call the resulting groups $\mathcal C$-graph automatic.  We prove that the class of $\mathcal C$-graph automatic groups is
closed under change of generating set, direct   and free product for certain classes  $\mathcal C$.
 We show that for quasi-realtime counter-graph automatic
groups  where normal forms have length that is linear in the geodesic length, 
there is an algorithm to compute normal forms (and therefore solve the word problem) in polynomial time. 
The class of quasi-realtime counter-graph automatic groups  includes all
Baumslag-Solitar groups, and  the free group of countably infinite rank. Context-sensitive-graph automatic groups are shown
to be a very large class, which encompasses, for example, groups with unsolvable conjugacy problem, the Grigorchuk group,
and Thompson's groups $F,T$ and $V$. 
\end{abstract}

\maketitle

\section{Introduction}

In this article we consider  extensions of the notion of a graph automatic group, introduced by Kharlampovich,  Khoussainov and Miasnikov in \cite{\KKM}, replacing the regular languages in their definition by more powerful language classes. Primarily we focus on the classes of context-free, counter, indexed and context-sensitive languages. 
We find that replacing regular languages with (quasi-realtime) counter languages preserves many of the desirable properties that  graph automatic groups enjoy, including a polynomial time algorithm to compute normal forms.
We  prove  that a finitely generated group is deterministic context-sensitive-graph automatic (with {\em quasigeodesic normal form} as defined below)
precisely when its word problem is  deterministic context-sensitive. It follows that the class of such groups is very large, and encompasses, for example,   groups with unsolvable conjugacy problem, the Grigorchuk group, and Thompson's group $V$ and all of its subgroups, which include Thompson's groups $F$ and $T$.
We present several  examples of counter-graph automatic groups, including  the   non-solvable Baumslag-Solitar groups, which we show to be  3-counter-graph automatic.
In \cite{JenSharif} the authors and Sharif Younes prove that Thompson's group $F$ is counter-graph automatic.

Several authors have considered generalized versions of automatic groups using different automata in place of finite state machines: Bridson and Gilman  introduced a geometric version of {\em asynchronously} automatic groups using indexed languages \cite{\BridGilman}; Baumslag, Shapiro and Short defined a class based on parallel computations by pushdown automata \cite{\polypush};
 and  Cho considered a version with counter languages in his PhD thesis \cite{\Cho}. Recent work of Brittenham and Hermiller \cite{\BritHerm} introduces the class of {\em autostackable groups} which also generalize the notion of automaticity.

The article is organized as follows. In Section \ref{sec:defns} we define the key notions of counter languages and $\mathcal C$-graph automatic groups  used in the paper. In  Section \ref{sec:counterpoly} we give a polynomial time  algorithm which computes normal forms in counter-graph automatic groups, and in Section \ref{sec:cs} we examine the consequences of permitting context-sensitive languages in the definition of $\mathcal C$-graph automatic groups.  In
Section \ref{sec:closure} we consider
 closure properties of $\mathcal C$-graph automatic groups, and in Section \ref{sec:egs}
we give examples of groups with counter-graph automatic structures.

Many of the ideas in this paper come  from the paper by  Olga Kharlampovich,   Bakhadyr Khoussainov and  Alexei Miasnikov \cite{\KKM}, and we are grateful for their help with this project.
We also thank Bob Gilman, Pascal Weil and especially Sharif Younes  for helpful conversations about this paper.
Lastly we thank the anonymous referee for  helpful feedback and suggestions.

\section{Background and Definitions}\label{sec:defns}

\subsection{Languages and automata}
For standard definitions of finite state, pushdown, nested stack, and linear bounded automata (accepting regular, context-free, indexed and context-sensitive languages respectively) see, for example, \cite{\HopUl}.
We begin by defining the particular types of counter automata we will use.

\subsubsection{Counter automata}
There are many variants of counter automata and languages in the literature,
see for example \cite{\BookGins, \EGaut, \ElderBS, \EKO, \Fischer, \GreibA, \GreibB, \Kambites, \Mitrana}.
In this article we define a counter automaton as follows.

\begin{defn}[counter automaton]  A counter automaton can be defined with a variety of attribues:

\medskip

\begin{enumerate}
\item A {\em blind deterministic $k$-counter automaton} is
  a  deterministic finite
state automaton augmented with a finite number of integer counters: these are all
initialized to zero, and can be incremented and decremented during operation,
but not read; the automaton accepts a word exactly if it  reaches
an accepting state with the counters all returned to zero. \footnote{These are called {\em $\mathbb Z^k$-automata} in \cite{\EKO,\Mitrana,\KambitesFormal}.}

\medskip

\item A {\em non-blind determistic $k$-counter automaton} is
  a deterministic finite
state automaton augmented with a finite number of integer counters: these are all
initialized to zero, and can be incremented, decremented, compared to zero and set to zero during operation;  the automaton accepts a word exactly if it  reaches
an accepting state with the counters all returned to zero.

\medskip

\item A (blind or non-blind) $k$-counter automaton is {\em non-deterministic} if from each state there can be multiple transitions  labeled by the same input letter, and transitions that read no input letter, labeled by $\e$. \footnote{These are called {\em multi-stack-counter automata} in  \cite{\BookGins}.} Following Book and Ginsburg \cite{\BookGins}  we require these automata to run in {\em quasi-realtime}, meaning there is a  bound on the number of consecutive  $\e$ transitions allowed.
\end{enumerate}

Define $\mathscr S_k$ to be the class of languages accepted by a non-blind  non-deterministic $k$-counter automata running in quasi-realtime, and $\mathscr C_k$ to be the class of languages accepted by a blind non-deterministic $k$-counter automata running in quasi-realtime.
 \end{defn}

 It is well known (\cite{\Minsky}, see also \cite{\HopUl} Theorem 7.9) that  a non-blind  non-deterministic $k$-counter automata  with $k\geq 2$  and no time restriction
can simulate a Turing machine, and so the class of languages accepted by such automata coincides with the class of recursively enumerable languages.
  Book and Ginsburg \cite{\BookGins}
 prove that imposing the quasi-realtime requirement, the languages $\mathscr C_k$ and $\mathscr S_k$ form a strict hierarchy:
\begin{thm}[Book and Ginsburg \cite{\BookGins}]
\label{thm:bookgins}
The language classes $\mathscr C_i$ and $\mathscr S_i$ satisfy the following inclusions.
\begin{eqnarray}
 \mathscr C_1\subsetneq \mathscr S_1\subsetneq \mathscr C_2 \subsetneq \mathscr S_2\subsetneq \mathscr C_3 \subsetneq \cdots 
\end{eqnarray}\end{thm}
 %%They  call non-blind non-determinstic $k$-counter automata {\em multi-stack counter acceptors}.

 In this article all counter automata are assumed to run in quasi-realtime.

\begin{lem}\label{lem:partialnondet}
If $L\in \mathscr S_k$ then   there is a constant $F$ so that on reading a word of length $n$ the absolute value of any counter is at most $Fn$.
\end{lem}\begin{proof}
Let $M$ be the non-deterministic $k$-counter automaton accepting $L$, and
suppose the maximum amount any counter is changed by any transition is $m$.
On  input $u=u_1\dots u_n$
consider all paths in $M$   labeled
$e_0 u_1e_1\dots e_{n-1}u_ne_n$
where $e_i$ is a string of $\e$ transitions, which by assumption has length at most some bound $D$.
Then each subpath $e_i$ can change the value of a counter by at most $Dm$, and so the entire path can change a counter by at most $Dm(n+1)+nm<3Dmn$, so set $F=3Dm$.
\end{proof}

\begin{cor}
The classes  $\mathscr C_k, \mathscr S_k$ are strictly contained in the class of non-deterministic context-sensitive languages.
\end{cor}
\begin{proof}
A $k$-counter automaton can be simulated by a Turing machine, with each counter value stored on the tape. On input of length $n$, the amount of tape required to store the values of all counters is 
$kFn$ by Lemma~\ref{lem:partialnondet}.
The containment is strict by Theorem~\ref{thm:bookgins}.
\end{proof}

% Book and Ginsburg call non-blind non-determinstic $k$-counter automata {\em multi-stack counter acceptors}. They remark that the languages accepted by multi-stack counter
%acceptors with at least two counters and no time restriction
%coincide with the recursively enumerable sets, that is, 
% multi-stack counter acceptors (with two or more counters) have the same computing power as a Turing machine.
% They prove that imposing the quasi-real time requirement, the languages $\mathscr C_k$ form a strict hierarchy

In drawing $k$-counter automata (see examples in Section \ref{sec:egs}) we label transitions by the input letter to be read, with subscript a $k$-tuple from the following alphabet:
\begin{itemize}
\item $+,-$ to increase/decrease a counter by 1
\item $+m,-m$ to increase/decrease a counter by $m\in\mathbb N$
\item $=,\neq$ to compare a counter to zero
\item $\downarrow$ to set a counter to zero.
\end{itemize} For example, in a non-blind $4$-counter automaton
 the label $1_{+,\neq\downarrow,,-3}$ means if the second counter is not 0, read input letter $1$, add 1 to the first counter,
  set the second counter to 0, make no change to the third counter, and subtract 3 from the last counter;  if the second counter was 0 then the transition is not followed.

\subsection{Closure properties of formal language classes}

We briefly outline some closure properties of the formal language classes we consider below.

\begin{defn}[homomorphism of languages]   %; \cite{\HopUl}]
Let $\Lambda, \Sigma$ be finite alphabets. For each $\lambda\in \Lambda$ let $r_{\lambda}\in \Sigma^*$ be a finite word, and let $L\subseteq \Lambda^*$.
Then $\phi:L\rightarrow \Sigma^*$ defined by $\phi(\lambda_1\dots \lambda_k)=r_{\lambda_1}\dots r_{\lambda_k}$ for $\lambda_i \in L$ is a {\em homomorphism of formal languages}.
If $r_{\lambda_i}$ is not the empty word for any $\lambda_i$ then $\phi$ is an {\em $\e$-free homomorphism}.
\end{defn}

A class $\mathcal C$ of formal languages is closed under ($\e$-free) homomorphism if $L\in\mathcal C$ is a language in the finite alphabet $\Lambda$ and $\phi:\Lambda^*\rightarrow \Sigma^*$ is any homomorphism, then $\phi(L) \in \mathcal C$.
The class $\mathcal C$ is closed under inverse homomorphism if for any $L \subset \Sigma^*$, where $\Sigma$ is any finite alphabet, and any homomorphism $\phi:\Lambda^*\rightarrow \Sigma^*$, if $L \in {\mathcal C}$ then $\phi^{-1}(L) \in {\mathcal C}$.

Closure of a formal language class ${\mathcal C}$ under finite intersection varies widely with ${\mathcal C}$.  The class of regular languages, for example, is closed under finite intersection, but the class of context-free languages is not, although the intersection of a context-free language and a regular language is again context-free.  %To amend the lack of closure for context-free languages,
In her thesis, Brough introduces the following class of languages.

\begin{defn}[poly-context-free;  \cite{\Tara}]
A language $L\subseteq \Sigma^*$ is {\em $k$-context-free} if it is the intersection of at most $k$ context-free languages, and {\em poly-context-free} if it is the intersection of some finite number of context-free languages.
\end{defn}

By design, the class of poly-context-free languages is closed under taking finite intersection, and intersection with regular languages.
%%Similarly we may define $k$- and poly-indexed languages.

%\begin{lem}
%If $L$ is accepted by a  blind $k$-counter  automaton, then $L$ is $k$-context-free. \end{lem}
%\begin{proof}
%Let $M$ be the $k$-counter automaton for $L$.
%For $1\leq i\leq k$ construct a pushdown automaton $M_i$ as follows. For each state $to have the same states as $M$, and for each transition $\tau$ to $\sigma$ labeled  $x_{(c_1,\dots, c_k)}$ in $M$, $M_i$ has a transition %from $\tau$ to $\sigma$ labeled $x$ with stack instructionas follows: if
%\end{proof}

The following lemma describes the closure of the class of counter languages under intersection.

\begin{lem} \label{lem:counterintersect}
The intersection of a $k$-counter  language with a regular language is $k$-counter, and the intersection of $k$- and $l$-counter languages is a $(k+l)$-counter language.  \end{lem}
\begin{proof}
Let $M$ and $N$ be counter automata with $k$ and $l$ counters  respectively. Define a $(k+l)$-counter automaton  with states $S\times T$ where $S$ are the  states  of $M$ and $T$ are the states of $N$, as follows.
Put a transition from $(s,t)$ to $(s',t')$ labeled by $\lambda_{\mathbf x}$ if
\begin{itemize}
\item  there is a transition from $s$ to $s'$ in $M$ labeled $\lambda_{(x_1,\dots, x_k)}$,
 \item  there is a transition from $t$ to $t'$ in $N$ labeled $\lambda_{(y_1,\dots, y_l)}$, and
 \item $\mathbf x=(x_1,\dots, x_k, y_1,\dots, y_l)$\end{itemize}
  where $x_i,y_j$ are counter instructions.

If $l=0$ then $N$ is simply a finite state automaton and we recover the first statement.
Note that the resulting automaton is blind and/or deterministic if and only if both $M$ and $N$ are.
\end{proof}

A linear bounded automaton  is a  Turing machine with memory linearly bounded by the size of the input, that is, there is a constant $E$ so that  on input a word of length $n$, the number of squares on the tape that can be used is $En$.  See, for example, \cite{\HopUl}.
In this article a  language is {\em (deterministic) context-sensitive} if it is the set of strings  accepted by a (deterministic) linear bounded automaton. With this definition a context-sensitive language can contain the empty string.   See \cite{\HopUl} (pp. 225--226) and \cite{\Shapiro} for a discussion of this.

\begin{lem}\label{lem:langsclosed}
The classes of   regular, counter, and  poly-context-free languages are closed under
 homomorphism, inverse homomorphism, intersection with regular languages, and finite intersection.

 The class of context-sensitive languages is closed under
$\e$-free homomorphism, inverse homomorphism, intersection with regular languages, and finite intersection.
\end{lem}
\begin{proof}
See Chapter 11 of \cite{\HopUl} for the cases of regular and context-sensitive languages, \cite{\Tara} for  poly-context-free languages, and \cite{\GreibGins} for counter languages.
\end{proof}

\subsection{${\mathcal C}$-graph automatic groups}
Let $G$ be a group with symmetric  generating set $X$, and $\Lambda$ a finite set of symbols. In general we do not assume that $X$ is finite.
The number of symbols (letters) in a word $u\in\Lambda^*$ is denoted $|u|_{\Lambda}$.
\begin{defn}[quasigeodesic normal form]
A  \textit{normal form for $(G,X,\Lambda)$} is a set of words $L\subseteq \Lambda^*$ in bijection with $G$. A normal form $L$ is
  \textit{quasigeodesic}  if there is a constant $D$ so  that  for each $u\in L$, $|u|_{\Lambda}\leq D(||u||_X+1)$ where $||u||_X$ is the length of a geodesic  in $X^*$ for the group element represented by  $u$.
\end{defn}
The $||u||_X+1$ in the definition allows for normal forms where the identity of the group is represented by a nonempty string of length at most $D$.  We denote the image of $u\in L$ under the bijection with $G$ by $\ou$.

Next we define the {\em convolution} of strings, which will be needed throughout the paper.
%If $\Lambda$  is a set and $k$ a positive integer then we write elements of  $\Lambda^k$ as column vectors \[\left(\begin{array}{c}\lambda_1\\\vdots\\\lambda_k\end{array}\right)\]where $\lambda_i\in\Lambda$.

\begin{defn}[convolution;  Definition 2.3 of \cite{\KKM}]
Let $\Lambda$ be a finite set of symbols, $\diamond$  a symbol not in $\Lambda$,  and let $L_1,\dots, L_k$ be a finite set of languages over $\Lambda$. Put  $\Lambda_{\diamond}=\Lambda\cup\{\diamond\}$. Define the {\em convolution of a tuple} $(w_1,\dots, w_k)\in L_1\times \dots \times L_k$ to be the string $\otimes(w_1,\dots, w_k)$ of length $\max |w_i|_{\Lambda}$ over the alphabet $\left(\Lambda_{\diamond}\right)^k$  as follows.
The $i$th symbol of the string is
\[\left(\begin{array}{c}
\lambda_1\\
\vdots\\
\lambda_k
\end{array}\right)\]
where $\lambda_j$ is the $i$th letter of $w_j$ if $i\leq |w_j|_{\Lambda}$ and $\diamond$ otherwise.
Then  \[\otimes(L_1,\dots, L_k)=\left\{\otimes(w_1,\dots, w_k) \mid w_i\in L_i\right\}.\]
\end{defn}
As an example, if $w_1=aa, w_2=bbb$ and $w_3=a$ then
\[\otimes(w_1,w_2,w_3)=\left(\begin{array}{c}
a\\
b\\
a
\end{array}\right)
\left(\begin{array}{c}
a\\
b\\
\diamond
\end{array}\right)
\left(\begin{array}{c}
\diamond\\
b\\
\diamond
\end{array}\right)\]

When $L_i=\Lambda^*$ for all $i$ the definition in \cite{\KKM} is recovered.

We begin with the definition of an automatic group, as introduced in \cite{\WordProc}. %%%Not exactly, need to say uniqueness issue

\begin{defn}[automatic group;   \cite{\WordProc}]
Let $(G,X)$ be a group and symmetric finite generating set.
We say that $(G,X)$ is \textit{automatic}
if there is a regular normal form $L\subseteq X^*$ such that for each $x\in X$ the language
$$L_x=\{\otimes(u,v) \mid u,v\in L,  \ov=_G \ou x\}$$ is regular.
\end{defn}
We remark that the usual definition of an automatic group requires a regular language $L$ to be in surjection with $G$, rather that in bijection. Theorem 2.5.1 of \cite{\WordProc} tells us that if a group has an automatic structure then there is an alternate automatic structure with a unique normal form word for each group element. Hence there is no loss of generality in requiring a normal form to be in bijection with the group.

Kharlampovich,  Khoussainov and Miasnikov extended this definition in \cite{\KKM}
by allowing the language of normal forms to be defined over a finite alphabet other than a generating set for the group.

\begin{defn}[graph automatic group;  \cite{\KKM}]
Let $(G,X)$ be a group and symmetric generating set, and $\Lambda$ a finite set of symbols.
We say that $(G,X,\Lambda)$ is \textit{graph automatic}
if there is a regular normal form $L\subseteq \Lambda^*$ such that for each $x\in X$ the language $$L_x=\{\otimes(u,v) \mid u,v\in L,  \ov =_G \ou x\}$$ is regular.
\end{defn}
Note that unlike \cite{\KKM} we do not insist that the generating set $X$ be finite; again our definition of a normal form requires a bijection between the group elements and the language of normal forms.

A useful first example to consider is the Heisenberg group (Example 6.6 of \cite{\KKM}), which is not automatic as it has a cubic Dehn function, but is graph automatic.
To prove the latter statement, matrices are represented as the convolution of three binary integers.
%Elements of the group are upper triangular integer matrices, which can be encoded as  the convolution of three binary integers.

The class of graph automatic groups includes the following groups which are known not to be automatic: the solvable Baumslag-Solitar groups, class 2 nilpotent groups,  % (including the Heisenberg group above),
 and non-finitely presented groups \cite{\KKM}. It also includes groups with unsolvable conjugacy problem \cite{\MSunic}.
It is not known if groups of intermediate growth belong to this class. Miasnikov and Savchuk \cite{\Dima} have shown that certain {\em graphs} of intermediate growth are graph automatic; see \cite{\KKM} for the definition of automatic structures on objects other than groups.

In this article we further extend the notion of a graph automatic group by replacing regular languages with other formal language classes.

\begin{defn}[$\mathcal C$-graph automatic group] \label{def:Cgraph}
Let $\mathcal B$ and $\mathcal C$ be formal language classes, $(G,X)$ a group and symmetric  generating set, and $\Lambda$ a finite set of symbols.
\begin{enumerate}
\item We say that $(G,X,\Lambda)$ is \textit{$(\mathcal B,\mathcal C)$-graph automatic}
if there is a  normal form $L \subset \Lambda^*$ in the language class $\mathcal B$, such that for each $x\in X$ the language 
$$L_x=\{\otimes(u,v) \mid u,v\in L,  \ov =_G \ou x\}$$ is in the class $\mathcal C$.
\item If $\mathcal B=\mathcal C$ then we say that $(G,X,\Lambda)$ is \textit{$\mathcal C$-graph automatic}.
%If $\Lambda=X$ then we say that $(G,X)$ is \textit{$(\mathcal B,\mathcal C)$-automatic}.
\item If   $\mathcal B=\mathcal C$ and $\Lambda=X$   then we say that $(G,X)$ is \textit{$\mathcal C$-automatic}.
\end{enumerate}
\end{defn}
For each $x \in X$ let $M_x$ denote the automaton which accepts the language $L_x$.

In general we will restrict our attention to $\mathcal C$-graph automatic groups, where $\mathcal C$ is one of the following language classes: context-free; indexed; context-sensitive; poly-context-free;
and (quasi-realtime) counter.
As checking membership in $L_x$ includes verifying that each of $u,v$ in $\otimes(u,v)$ lie in $L$, the complexity of the class $\mathcal C$ is in general greater than or equal to that of $\mathcal B$. Precisely:

\begin{lem}
 If $\mathcal C$ is closed under homomorphism, then a $(\mathcal B,\mathcal C)$-graph automatic group is $\mathcal C$-graph automatic.\end{lem}
\begin{proof} Define a homomorphism from $\otimes(L,L)$ to $L$ by a map that sends $ {\lambda_1\choose\lambda_2}$ to $ \lambda_1$ and  $  {\diamond\choose \lambda_1}$ to $\e$ for all  $\lambda_1\in \Lambda$ and $\lambda_2\in \Lambda_{\diamond}$.
Then the language $L$ is the image of $L_x$ under this homomorphism restricted to $L_x$, so is in $\mathcal C$.
\end{proof}

\begin{cor}
If
  $\mathcal B$ and $\mathcal C$ are each  one of  the classes of regular,  poly-context-free, quasi-realtime counter, or context-sensitive languages, then a $(\mathcal B,\mathcal C)$-graph automatic group is $\mathcal C$-graph automatic.
\end{cor}
\begin{proof}
Since each class is contained within the class of context-sensitive languages, if $\mathcal C$ is context-sensitive then the result follows.  Otherwise $\mathcal C$ is closed under homomorphism and the lemma applies.
\end{proof}

Definition \ref{def:Cgraph} extends naturally to the context of biautomatic groups.

\begin{defn}[$\mathcal C$-graph biautomatic group]%; \  \cite{\WordProc}]
\label{def:biauto}
Let $\mathcal C$ be a formal language class, $(G,X)$ a group and symmetric finite generating set, and $\Lambda$ a finite set of symbols.
We say that $(G,X,\Lambda)$ is \textit{$\mathcal C$-graph biautomatic}
if there is a  normal form $L \subset \Lambda^*$ in the language class $\mathcal C$, such that for each $x\in X$ the languages  $\{\otimes(u,v) \mid u,v\in L,  \ov =_G \ou x\}$ and $\{\otimes(u,v) \mid u,v\in L,  \ov =_G x \ou\}$ are in the class $\mathcal C$.  If
 $\Lambda=X$ we say that $(G,X)$ is \textit{$\mathcal C$-biautomatic}.
\end{defn}
Miasnikov and \v{S}uni\'c \cite{\MSunic} show that the classes of graph automatic  and graph biautomatic groups are distinct. In Section \ref{sec:cs} we show that when $\mathcal C$ denotes the class of determinisitic-context-sensitive languages, the classes of $\mathcal C$-graph automatic and $\mathcal C$-biautomatic groups coincide. In addition, there are deterministic context-sensitive-biautomatic groups with unsolvable conjugacy problem, in contrast to the cases of biautomatic and graph biautomatic groups.  %%Jen -- further evidence that C=context sensitive is TOO FAR

In the proof of (\cite{\KKM}, Lemma 8.2) is  the following observation that graph automatic groups naturally possess a quasigeodesic normal form.
For completeness we %; in several arguments below we must require that certain ${\mathcal C}$-graph automatic groups have such a normal form as well. Hence we
 include a  proof of this observation.

\begin{lem}\label{lem:linearnf}
If $(G,X,\Lambda)$ is graph automatic with respect to the regular normal form $L$, then $L$ is a quasigeodesic normal form.
\end{lem}
\begin{proof}
Let $C$ be an integer  that is at least the length of the normal form for the identity, and at least the number of states in any of the finite state automata $M_x$, where $x \in X$.

Let $w=w_1\ldots w_n$ be a geodesic where $w_i\in X$, and let $u_i$ be the normal form word for the prefix $w_1\dots w_i$ of $w$, for $i=0,\dots, n$, with $u_0$ representing the identity.
By assumption $u_0$ has length at most $C$.

Assume for induction that the length of  $u_{i-1}$ is at most $Ci$.

The automaton $M_{w_i}$
accepts the string labeled $\otimes(u_{i-1}, u_i)$. If $u_i$ has length more than $C(i+1)$ then we have
$$\otimes(u_{i-1}, u_i) =
\left(\begin{array}{c}y_1\\v_1\end{array}\right)\left(\begin{array}{c}y_2\\v_2\end{array}\right)\dots \left(\begin{array}{c}y_m\\v_m\end{array}\right) \left(\begin{array}{c}\diamond\\v_{m+1}\end{array}\right)\dots \left(\begin{array}{c}\diamond \\ v_n\end{array}\right)
$$
where $m\leq Ci$ and $n>C(i+1)$, so $n-m>C$ which is more than the number of states in $M_{w_i}$.  If we apply the pumping lemma for regular languages to the suffix of $\otimes(u_{i-1},u_i)$ beginning with $\left( \begin{array}{c} \diamond \\ v_{m+1} \end{array} \right)$, we see that $M_x$ accepts infinitely many normal form expressions for $u_i$, contradicting the uniqueness of the normal form.
\end{proof}

Note that when we generalize to
$\mathcal C$-graph automatic groups, the lemma is no longer true --- in Section \ref{sec:egs} we give an example of a quasi-realtime 3-counter-graph automatic structure for the Baumslag-Solitar groups $BS(m,n)$ with non-quasigeodesic normal form.

Note that when proving a triple $(G,X,\Lambda)$ is $\mathcal C$-graph automatic, the following observation shows that it suffices to check that  just one of $L_x$ or $L_{x^{-1}}$ lies in the class $\mathcal C$ for each $x\in X$.
\begin{lem}
If $\mathcal C$ is closed under $\e$-free homomorphism,
then
$L_{x}\in \mathcal C$ if and only if $L_{x^{-1}}\in\mathcal C$.
\end{lem}
\begin{proof}
The homomorphism that  replaces each $ {\lambda_1\choose \lambda_2}$ by $ {\lambda_2\choose \lambda_1}$ for all $\lambda_i\in\Lambda_{\diamond}$ in $L_x$ yields the language $L_{x^{-1}}$.
\end{proof}

\subsection{Remarks on the definition of graph automatic groups}

In \cite{\KKM} the authors implicitly assume that the normal form for the identity element is always the empty string --- see, for example, the proof of Theorem 10.8 in \cite{\KKM}.
In generalizing their definition and results, we realized this was a subtle issue.
The definition of an automatic structure for a group $(G,X)$ asserts the existence of a bijection (or surjection) from $L\subseteq X^*$ to $G$, together with a finite collection of regular languages which have finite descriptions either in terms of regular expressions, finite state automata, regular grammars or otherwise. In this definition there is no explicit information about the bijection from $L$ to $G$, in particular the normal form word for the identity is not fixed by this. In Theorem 2.3.10 in \cite{\WordProc}, an algorithm is given that computes the normal form of any word in an automatic group, necessarily written in terms of the group generators, which runs in quadratic time.  At the end of the proof of Theorem 2.3.10, it is explained how this algorithm can be used (in constant time) to find the normal form word for the identity, thus making this algorithm constructive. Hence in the case of automatic groups, the definition alone is enough to construct the bijection from $L$ to $G$.

In the case of a graph automatic or $\mathcal C$-graph automatic group $(G,X,\Lambda)$, many analogous results are not constructive unless one knows at least one pair $q\in L\subseteq \Lambda^*$ and $p\in G$ with $\overline{q}=_G p$.  Hence this assumption is often included in the statement of the theorems in this paper.

We have modified the original definition of a graph automatic group by removing the requirement that $G$ be finitely generated. In the case of $\mathcal C$-graph automatic groups, this allows us to capture groups such as $F_{\infty}$ (see Section \ref{sec:egs}). Since $\Lambda$ is finite, $G$ must be countable.
We were not able to find an example of a countably infinitely generated graph automatic group, so our evidence justifying this modification is perhaps less convincing.
We add the hypothesis that $G$ is finitely generated in several statements below  on counter and context-sensitive-graph automatic groups.

Finally we remark that we know of no examples of $\mathcal C$-graph automatic groups which we can prove not to be graph automatic. This paper (and the examples we present in Section \ref{sec:egs} and in \cite{\papertwo})
grew out of an attempt to  decide whether or not examples  such as non-solvable Baumslag-Solitar groups and R. Thompson's group $F$ are graph automatic.

\section{Counter-graph automatic implies polynomial time algorithm to compute normal forms}\label{sec:counterpoly}

In this section we extend the results of Epstein {\em et al.} (\cite{\WordProc} Theorem 2.3.10) and Kharlampovich {\em et al.} (\cite{\KKM} Theorem 8.1)
to show that for any finitely generated  $\mathscr S_k$-graph automatic group there is an algorithm to compute normal forms for group elements that runs in polynomial time. 
Recall that $\mathscr S_k$ denotes the class of languages accepted by a non-deterministic, quasi-realtime, non-blind k-counter automaton; this class includes languages accepted by blind and/or deterministic $k$-counter languages.

\begin{thm}\label{thm:counterpoly}
Let $(G,X)$ be a group with finite  symmetric generating set, and $\Lambda$  a finite set of symbols so that $(G,X,\Lambda)$ is $\mathscr S_k$-graph automatic with  quasigeodesic normal form $L$. Moreover, assume we are given $p\in X^*$ and $ q\in L$ with $p=_G \overline{q}$.
Then there is an algorithm that, on input a word $w=x_1\dots x_n\in X^*$, computes $u\in L$ with $\overline{u}=_G w$, which runs in time $O(n^{2k+2})$.
\end{thm}

\begin{proof}
We will give an algorithm that on input $w=x_1x_2 \cdots x_r \in X^*$ computes $u\in L$ where $\ou=_G pw$, which runs in time $O(r^{2k+2})$. Running this algorithm on input $p^{-1}$ gives a word  $\mu\in L$ so that $\overline{\mu}=_Ge$. The final algorithm is obtained with $q=\mu$ and $p=e$. Since $p^{-1}$ has a fixed length, applying the algorithm to  compute $\mu$ takes constant time.

For each $x\in X$ let  $M_x$ be the non-deterministic $k$-counter automaton accepting the language $\{\otimes(u,v) \mid u,v\in L, \ov =_G \ou x\}$ in quasi-real time.  We begin with an enumeration of constants which appear in this argument.
\begin{enumerate}
\item Let $C$ be the quasigeodesic normal form constant for $L$.
\item Let $D$ be the maximum number of states in any $M_x$.
\item Let $E$ be the maximum over all  $M_x$ of the in-degree or out-degree of any vertex.
\item Let $F$ be the maximum over all $M_x$ of the constant in Lemma \ref{lem:partialnondet}; so on input of length $n$, the maximum absolute value of any counter in any machine $M_x$ is $Fn$.
\item Let $K-1$ be the maximum number of consecutive $\e$ edges that can be read in any $M_x$.
\item Let $P=|p|_X$ be the length of the word $p\in X^*$.
\end{enumerate}
Note that we require finitely many generators to guarantee the existence of the constants $D,E$ and $F$.

For each $i\in [1,n]$,  let $u_i\in L$ be the string such that $\overline{u_i}=_G px_1\dots x_i$, and set $u_0=q$, so $\overline{u_0}=_G p$.
Assume for induction that we have computed  and stored $u_i$ in time $O(i^{2k})$. Since $u_0=q$ is constant length, the claim is true for $i=0$.
 We find $u_{i+1}$ in time $O((i+1)^{2k+1})$ as follows.

Write $u_i=\kappa_1\ldots \kappa_s\in L$ with $\kappa_j\in \Lambda$, and note that  since  $\overline{u_i}=_G px_1\dots x_i$ we have $s\leq C(P+i+1)$.
Let $M=M_{x_{i+1}}$ be the non-deterministic $k$-counter automaton accepting $\otimes(u_i,u_{i+1})$.

Define a {\em configuration} of $M$ to be a pair $(\tau, \mathbf{c})$ where $\tau$ is a state of $M$ and $\mathbf{c}\in \mathbb Z^k$ represents the value of each counter.
If $\tau_0$ is the start state for $M$, then $(\tau_0, \mathbf 0)$ is the {\em start configuration} where $\mathbf 0=(0,\dots, 0)$.
Let  $(\tau, \mathbf{c})_{\diamond}$ denote a configuration of $M$ which is obtained by reading an input string
 of the form $$\left(\begin{array}{c} \kappa_1 \\ \sigma_1\end{array}\right) \dots \left(\begin{array}{c}\kappa_l \\ \sigma_l\end{array}\right)\left(\begin{array}{c} \kappa_{l+1} \\ \diamond \end{array}\right) \dots  \left(\begin{array}{c}\kappa_s \\ \diamond\end{array}\right)$$ where $\sigma_t\in \Lambda$  and $l<s$, that is, the length of the string of symbols in the top coordinate is strictly longer than then length of the string of symbols in the bottom coordinate.

 If $\mathbf y$ is a $k$-array of counter instructions and  $\mathbf c$ is a $k$-tuple of counters, the notation  $\mathbf y(\mathbf c)$ means the $k$-tuple of  counter values after $\mathbf y$  is applied to  $\mathbf c$.  If $\omega$ is a finite path in $M$ let $[\omega]_{\mathbf y}$ denote the path with  all the counter instructions collected together as $\mathbf y$.

We now build a directed graph ${\mathcal G}$ with vertices and edges defined recursively as follows. Vertices will be grouped together in sets $S_j$, and edges in sets $T_j$. The set $S_j$ will consist of all configurations that can be obtained from the start configuration by following a path in $M$ which contains exactly $j$ edges not labeled by $\e$. For $j\leq s$,  $S_j$ is the set of configurations of $M$ that can be obtained by reading $$\left(\begin{array}{c} \kappa_1 \\ \sigma_1\end{array}\right) \dots \left(\begin{array}{c}\kappa_j \\ \sigma_j\end{array}\right)$$ where $\sigma_t\in \Lambda_{\diamond}$.

%An edge in $T_j$ is of the form $((\tau_1,\mathbf c_1), (\tau_2,\mathbf c_2), \lambda)$ where $(\tau_1,\mathbf c_1)\in S_{j-1},  (\tau_2,\mathbf c_2)\in S_j, \lambda\in\Lambda_\diamond$ and is defined when there is a path in $M$ labeled $\e^r {\kapp\lambda$

The set $S_0$ consists of the configuration $(\tau_0, \mathbf 0)$, together with all configurations that can be reached by reading a path labeled $\e^k$
 from the start state in $M$. Recall that the number of consecutive $\e$ transitions is bounded, so the set $S_0$ can be constructed by searching a bounded number of paths. Precisely, we must check at most $$\sum_{k=1}^{K-1}E^{k}=O(E^K)$$ paths.
 % where since $E$ is the maximum outdegree in $M$ and $K-1$ is the maximum number of consecutive $\e$ edges in $M$.

 Given $S_j$ with $j<s$, we construct $S_{j+1}$ together with  the set $T_{j+1}\subseteq S_j\times S_{j+1}\times \Lambda_{\diamond}$ of directed edges as follows.
\begin{enumerate}\item Initially set $S_{j+1} = T_{j+1} = \emptyset$.
\item
For each $(\tau, \mathbf c)\in S_j$ and each path from $\tau$ to $\tau'$ in $M$ labeled $\left[\left(\begin{array}{c} \kappa_{j+1} \\ \sigma\end{array}\right)\e^r\right]_{\mathbf y}$ with $\sigma\in \Lambda$ and $\mathbf y$ a $k$-array of counter instructions, add  $(\tau', \mathbf y(\mathbf c))$ to $S_{j+1}$, and add $((\tau, \mathbf c),(\tau', \mathbf y(\mathbf c)),\sigma)$ to $T_{j+1}$.

\item For each $(\tau, \mathbf c)\in S_j$ and each path from $\tau$ to $\tau'$ in $M$ labeled $\left[\left(\begin{array}{c} \kappa_{j+1} \\ \diamond\end{array}\right)\e^r\right]_{\mathbf y}$, add  $(\tau', \mathbf y(\mathbf c))_{\diamond}$ to $S_{j+1}$, and add $((\tau, \mathbf c),(\tau', \mathbf y(\mathbf c))_{\diamond},\diamond)$ to $T_{j+1}$.

\item For each $(\tau, \mathbf c)_{\diamond}\in S_j$ and each path from $\tau$ to $\tau'$ in $M$ labeled $\left[\left(\begin{array}{c} \kappa_{j+1} \\ \diamond\end{array}\right)\e^r\right]_{\mathbf y}$, add  $(\tau', \mathbf y(\mathbf c))_{\diamond}$ to $S_{j+1}$, and add $((\tau, \mathbf c)_{\diamond},(\tau', \mathbf y(\mathbf c))_{\diamond},\diamond)$ to $T_{j+1}$.
\end{enumerate}
Since the number of consecutive $\e$ transitions in $M$ is at most $K-1$, that is, $0\leq r\leq K-1$,  the counter instructions $\mathbf y$ above are bounded.

Any configuration appearing in $S_j$ and $T_j$ is one that can be reached by reading $\otimes(\kappa_1\dots \kappa_j, v)$ for some $v\in\Lambda_{\diamond}^*$.
%Recall that $s = |u_i|_{\Lambda}$.
It follows that the set $S_s=S_{|u_i|_{\Lambda}}$ contains all possible configurations of $M$ that can be reached by reading any string $\otimes(u_i, v)$ where $v\in\Lambda_\diamond^*$. If $S_s$ does not contain a configuration $(\tau_a,\mathbf 0)$ or
$(\tau_a,\mathbf 0)_{\diamond}$
where $\tau_a$ is an accept state of $M$,
continue to construct sets $S_{j+1}$ and $T_{j+1}$ with $j\geq s$ as follows.
\begin{enumerate}
\item Remove all elements of $S_s$ of the form $(\tau,{\mathbf c})_{\diamond}$.  A path to such a configuration cannot be extended to an accept configuration.

% and so a path that endshence cannot be contained in a path which reaches an accept state in $M$.
%these paths cannot be contained in a path which must be extended to reach an accept state in M.

\item Set $j=s$.
\item While  $S_j$  does not contain a configuration   $(\tau_a,\mathbf 0)$
where $\tau_a$ is an accept state of $M$:
\begin{enumerate}\item
For each $(\tau, \mathbf c)\in S_j$ and each path from $\tau$ to $\tau'$ in $M$ labeled $ \left[\e^r   \left(\begin{array}{c} \diamond \\ \sigma\end{array}\right)\right]_{\mathbf y}$ with $\sigma\in \Lambda$ and $\mathbf y$ a $k$-array of counter instructions, add  $(\tau', \mathbf y(\mathbf c))$ to $S_{j+1}$, and add $((\tau, \mathbf c),(\tau', \mathbf y(\mathbf c)),\sigma)$ to $T_{j+1}$.
\item Increment $j$ by 1.
\end{enumerate}
\end{enumerate}
Since $L$ is a quasigeodesic normal form for $G$ and $\overline{u_{i+1}}=_G px_1\dots x_{i+1}$, the length of $u_{i+1}$   is bounded by $C(P+i+2)$. It follows that  $S_j$ will contain an accept configuration for some $j\leq C(P+i+2)$, at which point the loop stops.

The time to construct and store the sets $S_{j+1}$ and $T_{j+1}$  is computed as follows.
For each configuration in $S_j$ we check at most $E^K$ paths of length at most $K$ in $M$, where $K-1$ is the maximum number of consecutive $\e$ edges that can be read, and $E$ is the maximum  out-degree. So to compute and store $S_{j+1}$ and $T_{j+1}$
takes time $O\left(|S_j|E^K\right)$.

Let $m\in\mathbb N$ be the minimal value so that   $s\leq m\leq C(P+i+2)$ and $S_m$ contains an accept configuration $(\tau_a,\mathbf 0)$ or  $(\tau_a,\mathbf 0)_\diamond$ (in which case $m=s$).
As ${\mathcal G}$ is a directed graph, there is a directed labeled path $e_1\dots e_m$ where $e_j\in T_j$ from $(\tau_0,\mathbf 0)$ to $(\tau_a,\mathbf 0)$ or $(\tau_a,\mathbf 0)_\diamond$, which can be found by backtracking through ${\mathcal G}$, scanning edges in $T_j$ for $m\geq j\geq 0$.
%REDO THIS
%The labels along this string correspond to the required word $u_{i+1}$, which can be found by backtracking through the sets $T_j$ as follows.
%\begin{enumerate}
%\item Mark the configuration  $(\tau_a,\mathbf 0)$ in $T_m$ and set $j=m$.
%\item While $j>0$:
%\begin{enumerate}%
%\item Scan $T_j$ for a triple $(a,b,c)$ where $b$ is a marked configuration.
%\item If $c\neq \diamond$, print $c$.
%\item Mark the configuration $a$, and set $j=j-1$.
%\end{enumerate}
%\end{enumerate}
%The printed string is the required word $u_{i+1}$ written in reverse.
The time required to run this backtracking process is at most $O\left(\bigcup_{j=1}^m\left| T_j\right|\right)$.

 The time required to construct and store the sets $S_{j+1}$ and $T_{j+1}$ for $0\leq j<m$ is $O\left(\sum_{j=0}^{m-1} |S_j|E^K\right)$.
It follows that the total time complexity for the algorithm is
\[O\left(\sum_{j=1}^m |T_j|+E^K\sum_{j=0}^{m-1} |S_j|\right)=O\left(\sum_{j=1}^m \left(|T_j|+E^K |S_{j-1}|\right)\right)=O\left(\sum_{j=1}^m |T_j|\right)\] since $|S_{j-1}|\leq |T_j|$.

To complete the proof we compute $\sum_{j=1}^m |T_j|$.  If $(\tau,\mathbf c)\in S_j$ then $\tau$ can be one of $D$ states in $M$, and each counter has absolute value at most $Fj$ (so has value $c$ with  $-Fj\leq c\leq F_j$), so the number of possible configurations is $D(2Fj+1)^k$. We also have configurations of the form $(\tau,\mathbf c)_{\diamond}$, so $|S_j|\leq 2D(2Fj+1)^k$.

As $T_j\subseteq S_{j-1}\times S_j\times \Lambda_{\diamond}$ we have \[|T_j|\leq 2D(2F(j-1)+1)^k \cdot 2D(2Fj+1)^k \cdot (|\Lambda|+1)\leq Xj^{2k}\] where $X=X(D,F,k,|\Lambda|)$ is a fixed constant.
We also have $m\leq C(P+i+2)=Yi$ where $Y=Y(C,P)$ is a fixed constant. Thus
\[ \sum_{j=1}^{m} |T_j|\leq  \sum_{j=1}^{m} Xj^{2k}= X\sum_{j=1}^{m} j^{2k}\leq X\sum_{j=1}^{m} m^{2k}=Xm^{2k+1}\leq X(Yi)^{2k+1}=Zi^{2k+1}\]
where $Z=XY^{2k+1}=Z(C,D,F,P,k,|\Lambda|)$ is a fixed constant.

To compute $u_n$ which is the normal form for $pw$, we repeat this procedure for $i\in[1,n]$ so the total time complexity is $\sum_{i=1}^n Zi^{2k+1}\leq Zn^{2k+2}$.
\end{proof}

\section{Context-sensitive-graph automatic groups}
\label{sec:cs}

Recall that
a linear bounded automaton  is a  Turing machine together with a constant $D$ so that  on input a word $w$ of length $n$, the number of squares on the tape used for any operation involving $w$ is $Dn$.
The {\em read-head} of the Turing machine is a pointer to a particular square of the tape. A move of the Turing machine can involve reading the letter at the position of the read-head, writing to this position, or moving the read-head one square to the left or right.
A letter written on the tape can be {\em marked} by overwriting it with an annotated version of the letter --- for example the letter $a$ can be replaced by $\hat{a}$.

A language is context-sensitive if it is accepted by a linear bounded automaton, and \textit{deterministic context-sensitive}, or \dcs, if the linear bounded automaton is deterministic.
Note that here we allow content-sensitive languages to include the empty string ---  in some usages context-sensitive languages are defined without this, in particular when defined via a grammar in which the right-hand sides of production rules are required to have positive length.  Note also that is it not known if the class of deterministic and non-determistic linear space languages are distinct.

Shapiro \cite{\Shapiro} and Lakin and Thomas \cite{\LakinThomas, \LakinThomasGCC} consider groups with context-sensitive word problem. Shapiro showed that any finitely generated subgroup of an automatic group has \dcs\ word problem, and Lakin and Thomas proved several closure properties.

In this section we consider the class of \dcs-graph automatic groups. We show that if a finitely generated group $G$ has a \dcs-graph automatic structure with quasigeodesic normal form, then its word problem in solvable in deterministic  linear space. We also prove that if a finitely generated group $G$  has deterministic linear space word problem then it has a \dcs-biautomatic structure (with no symbol alphabet needed) with geodesic normal form language.

% It follows that the classes of finitely generated \dcs-graph automatic groups with quasigeodesic normal form coincides with the class of  finitely generated \dcs-biautomatic groups.

%The class of groups wih \dcs\ word probe
%This class includes many groups, including groups with unsolvable conjugacy problem.  It follows that \dcs-biautomatic does not imply solvable conjugacy problem.
%%{\em Murray - add reference for the above statement.}

We start with a simple subroutine to enumerate strings over an ordered alphabet in Shortlex order.  Recall that 
for a  finite totally ordered finite set $\Lambda$, the {\em Shortlex order} on $\Lambda^*$ is defined as follows: for  $u,v\in \Lambda^*$,  $u<_{\mathrm{SL}} v$  if 
\bi\item $|u|_{\Lambda} < |v|_{\Lambda}$, or
\item   $|u|_{\Lambda} = |v|_{\Lambda}$, $u=p\lambda_iu', v=p\lambda_jv'$ with $\lambda_i<\lambda_j$ and $p,u',v'\in\Lambda^*$.\ei

\begin{alg}[Shortlex subroutine]\label{alg:slsub}
%The following algorithm can be performed in linear space.
Let $\Sigma$ be a finite totally ordered set,   $\#,\$$ two symbols not in $\Sigma$, and  $\sigma_0,\sigma_r \in \Sigma$ such that $\$<\sigma_0\leq \sigma\leq \sigma_r$ for all $\sigma\in \Sigma$.
Let $v=v_1\dots v_k \in \Sigma^*$, and assume $\#v\$$ is written on the tape of a linear bounded automaton.
Then the next string in Shortlex order can be found and overwritten on the tape using space $k+2$ as follows.
\end{alg}
\begin{enumerate}
\item Move the read-head to the last letter of $v$ (before the $\$$ symbol), and set a boolean variable \texttt{done} to be false.
\item While not \texttt{done}:
\begin{enumerate}
\item If the letter  at the read-head position is $\sigma_r$, move the read-head one position to the left.
\item  If the read-head points to $\#$, the contents of the tape must be $\#\sigma_r^k\$$. In this case overwrite the tape by $\#\sigma_0^{k+1}$ (consuming the $\$$ symbol) and set \texttt{done} to be true.
\item Else  the letter  at the read-head position is $v_i\in \Sigma$ with $v_i<\sigma_r$.  The contents of  tape are
 $\#v_1\dots v_{i-1}v_i\sigma_r^{k-i}\$$. %% with $v_i<\sigma_r$.
  Let $v_i^*\in \Sigma$ be such that $v_i<v_i^*$ and $\sigma\leq v_i^*$ implies $\sigma\leq v_i$.  In this case  overwrite the tape by $\#v_1\dots v_{i-1}v_i^* \sigma_0^{k-i}\$$  and set \texttt{done} to be true.
\end{enumerate}
 \end{enumerate}

Note that the subroutine writes either $\#v'\$$ or $\#v''$ to the tape, where $|v'|_\Sigma=|v|_\Sigma$ and $|v''|_\Sigma=|v|_\Sigma+1$. If one ignores the $\#,\$$ symbols then the algorithm on input $v$ returns the next string in Shortlex order in $\Sigma^*$.

\begin{prop}\label{prop:biaut}
Let $G$ be a group and finite symmetric generating set $X$.
If  $(G,X)$ has \dcs\ word problem
then $(G,X)$ is  \dcs-biautomatic, with normal form the set of Shortlex geodesics over $X$.
\end{prop}
\begin{proof}
Assume the word problem algorithm for $(G,X)$ runs as follows. On input $u\in X^*$ written on a one-ended tape, the algorithm returns {\em yes} if $u$ is trivial and {\em no} otherwise, and returns a blank tape, using at most $D|u|$ space.

Fix an order on the generators with $x_0$ the smallest and $x_r$ the largest, and let $L$ be the set of Shortlex geodesic words for $G$ with respect to this order.
By Definition \ref{def:biauto} we must show that $L$ and the languages  $\{\otimes(u,v)  \mid u,v\in L, v=xu\}$  and $\{\otimes(u,v)  \mid u,v\in L, v=ux\}$ for each  $x\in X$ are \dcs.
Let $\$$ be a symbol not in $X$, and set $\$< x_0$.

Define a deterministic linear bounded automaton to accept $L$ as follows.
Assume that $\%,\#,\$$ are distinct symbols not in $X$.
On input  $u\in X^*$ of length $n$:
\begin{enumerate}
\item Write $\%u\#\left(\$\right)^{n+1}$ on the tape and set \texttt{done} to be false.
\item While not \texttt{done}:\begin{enumerate}
\item Set $v$ to be the word on the tape between  $\#$ and the first $\$$ symbol.
\item Scan the tape to check if $u$ and $v$ are identical as strings. If they are, accept $u$ and set \texttt{done} to be true.
\item Else write $uv^{-1}$ to the left of the  $\%$ symbol. Call
the word problem algorithm on the one-ended tape to the left of the $\%$ symbol. If it returns {\em yes},  reject $u$ and set \texttt{done} to be true \footnote{The contents of the tape after this step are $\%u\#v\left(\$\right)^i$ with $|v|_{\Lambda}+i=n+1$}.
\item Else run the Shortlex subroutine (Algorithm \ref{alg:slsub}) to overwrite $v$ by the next word in Shortlex order. \end{enumerate}
\end{enumerate}

The algorithm runs as follows. To start we have $v=\e$. If $u=v$ then the empty string is accepted since it is the Shortlex geodesic for the identity.
If not we overwrite $v$ with the next word in Shortlex order, and compare to $u$. We iterate the loop until either the contents of the tape are $\%u\#u\$$, or we find a word $v$ that equals $u$ in the group and is shorter in Shortlex order.  At any time the tape contains at most $4n+3$ letters, and running the word problem algorithm takes space at most $D|uv^{-1}|\leq D(2n)$, so all together the space required is $2Dn+4n+3$.

The following algorithm accepts 
\begin{align*}
\{\otimes(u,v)  \mid u,v\in L, v=xu\} \ (\mathrm{respectively} \ \{\otimes(u,v)  \mid u,v\in L, v=xu\})
\end{align*} for $x\in X$:
On input $\otimes(u,v)$, \begin{enumerate}
\item run the preceding algorithm on $u$ to check if $u\in L$;
\item run the preceding algorithm on $v$ to check if $v\in L$;
\item call the linear space word problem algorithm on  $uxv^{-1}$ (respectively $xuv^{-1}$).
\end{enumerate}
\end{proof}

%%{\em Question -- after the input, doesn't the machine first have to check that $u,v \in L$? Does that need to be part of this?} --- Yes this was before we realised $M_x$ had to check L !!
%{\em I don't think we need to state the obvious that you accept if the three algorithms give the correct answer.}

Note that there are subgroups of $F_2\times F_2$ with unsolvable conjugacy problem \cite{\Mihailova,\Miller}, which by \cite{\Shapiro} have \dcs\ word problem and therefore are \dcs-biautomatic. It follows that \dcs-biautomatic does not imply solvable conjugacy problem, in contrast to the graph biautomatic case (\cite{\KKM} Theorem 8.5).

Next we show that \dcs-graph automatic groups with quasigeodesic normal form have deterministic linear space word problem.

\begin{prop}
Let  $(G,X)$ be a group with  finite symmetric generating set, and $\Lambda$ a finite set of symbols so that $(G,X,\Lambda)$ is a \dcs-graph automatic group with  quasigeodesic normal form $L \subset \Lambda^*$. Additionally, suppose we are given $p\in X^*$ and $q\in L$ with $p=_G \overline{q}$.
Then there is an algorithm that on input a word $w=x_1\dots x_n\in X^*$, computes $u\in L$ with $\ou =_G w$ and runs in space $O(n)$.
\end{prop}

\begin{proof}
We first give the algorithm that on input $w\in X^*$ computes $u\in L$ where $\ou=_G pw$. Running this algorithm on input $p^{-1}$ gives a word  $\mu\in L$ for the identity. The final algorithm is obtained with $q=\mu$ and $p=e$. Since $p^{-1}$ has a fixed length the step to compute $\mu$ takes constant space.

 For each $x\in X$ let $L_x$ be the \dcs\ language $\{\otimes(u,v)  \mid u,v\in L,  \ov =_G \ou x\}$.
 We begin with an enumeration of constants which appear in this argument.
\begin{enumerate}
\item Let $B$ be a constant so that for any $x\in X$ the space used by the linear bounded automaton accepting $L_x$ on input of length $n$ is $Bn$.
\item Let $C$ be the quasigeodesic normal form constant for $L$.
\item Let $P=|p|_X$ be the length of the word $p\in X^*$.
\end{enumerate}
Note that we require finitely many generators to guarantee the existence of the constant $B$.

Let $w=x_1\ldots x_n\in X^*$ be the input word, and define  $w_0=p$, $w_i=px_1\ldots x_i$ for
 $i\in [1,n]$, and let $u_i\in L$ be such that $\overline{u_i}=_G w_i$.
Note that $u_0=q$, and for each $i$ the length of $u_i$ is at most $C(P+i+1)$.
Let $\#$  be a symbol not in $\Lambda$. Define a total order on the (finite)  set $\Lambda$.  %\cup \$$, with $\$$ the smallest element.

We compute the normal form word representing $w$ as follows.
Write $w\#u_0\#$ on the tape, marking the first letter of $w$. This uses space at most $n+2+C(P+1)$.  % (since $w$ and$\%,\#$ take space $n+2$, and $|q|\leq C(P+1)$
Assume for induction that we have written $w\#u_i\#$ on the tape  for $i<n$,  and  marked the letter at position $i+1$ in $w$, using space at most
$D(n) = n+2+(B+2)C(P+n+1)$.

Find $u_{i+1}$ as follows.
\begin{enumerate}
\item Set \texttt{done} to be false.
\item Let $v$ denote the string of symbols  to the right of the last $\#$ on the tape. To begin we have $v=\e$.
\item While not \texttt{done}:
\begin{enumerate}\item Run the deterministic linear space algorithm that accepts $L_{x_{i+1}}$ on $\otimes(u_i,v)$. Note that the length of the input to this subroutine is at most $C(P+n+1)$ since $L$ is quasigeodesic and  $u_i,u_{i+1}$ represent words of geodesic length at most $n$.
It follows that the space needed for this step is at most $BC(P+n+1)$.

\begin{enumerate}\item If the subroutine returns true, then we have found $v=u_{i+1}$. Set  \texttt{done} to be true. %If $rewrite the tape as $w\#u_{i+1}\#$ and either mark the letter in position $i+2$ of $w$ if $i+2<n$.
\item  Else run the Shortlex subroutine  (Algorithm \ref{alg:slsub})  to overwrite $v$ by the next word in Shortlex order.
\end{enumerate}
\end{enumerate}
If $i+1<n$, rewrite the tape as $w\#u_{i+1}\#$ and mark the letter at position $i+2$ of $w$.
If $i+1=n$, the word $u_n$ is the required normal form word for $w$.
\end{enumerate}

Since we know there is some string $u_{i+1}$  of length at most $C(P+i+2)$ then this algorithm must terminate.  Moreover, the amount of space used on the tape to store $w\#u_i\#v$ is bounded by $n+2+2C(P+n+1)$, as the length of $w\#\#$ is $n+2$, and $u_i,v$ have length at most $C(P+n+1)$. The space used to run the subroutine on $\otimes(u_i,v)$ is bounded by $BC(P+n+1)$, so in total the amount of space required is at most  $D(n)$.
\end{proof}

Combining these two propositions we obtain the following.

\begin{thm}
The following classes of groups coincide:
\begin{enumerate}
\item finitely generated \dcs-graph automatic groups with quasigeodesic normal form;
\item finitely generated \dcs-biautomatic groups with Shortlex geodesic normal form;
\item  finitely generated groups with \dcs\ word problem.
\end{enumerate}
\end{thm}

 The class
 of such groups  is very large --- groups with \dcs\ word problem  include
 all linear groups \cite{\LiptonZ},  {\em logspace embeddable} groups studied by the first author, Elston and Ostheimer \cite{\EEO}, 
and all finitely generated subgroups of automatic groups \cite{\Shapiro}. 
It also includes the co-indexed and co-context free groups as described in  \cite{\HoltCocf, \HoltCoindexed, \Joerg}. These groups have co-word problems accepted by  non-deterministic pushdown or nested-stack automata, which can be simulated by deterministic linear bounded automata since as described in these articles, the non-determinism is confined to an initial guessing step. It follows that the word problem for these groups is accepted by the same deterministic  linear bounded automata. These classes include  the Higman-Thompson groups, Thompson's group $V$, Houghton's groups, and the Grigorchuk group.

%%%By contrast, for example, blind deterministic $k$-counter-graph automatic groups in general will not have a blind deterministic $k$-counter word problem, since by \cite{\counterWP} the only %%%%groups with (blind) counter word problem are virtually abelian.
%%%  What are groups with $\mathscr S_k$ word problem?

Note that the number of  configurations of a linear bounded automaton is exponential in the length of the input string, so  the time complexity of computing the normal form of a word in a \dcs-biautomatic group is at most exponential.
The next example shows that  a polynomial time algorithm to compute normal forms of \dcs-biautomatic structures seems unlikely to exist.

Let $G=\mathbb Z_2\wr \mathbb Z^2$. By (\cite{\EEO} Theorem 14) the word problem for $G$ is in deterministic logspace and therefore deterministic  linear space, so it
follows from Proposition \ref{prop:biaut}
that $(G,X)$ is \dcs-biautomatic with Shortlex geodesic normal form, where X is the standard generating set. The {\em bounded geodesic length problem}  (see \cite{\GeodProbs, \MiasFreeSolv}) for a group $G$ with finite generating set $X$ is the following:

\begin{prob}
[Bounded geodesic length problem]
On input an integer $k$ and a string $w\in X^*$, decide if the geodesic length of $w$ is less than $k$.
 \end{prob}

Suppose one could prove that a \dcs-graph automatic structure with quasigeodesic normal form for a finitely generated group implied a polynomial time algorithm that on input a string of generators computes the normal form. Then by Proposition \ref{prop:biaut} we may assume the group has a  \dcs-biautomatic structure with normal form the set of all Shortlex geodesics.
 Parry \cite{\Parry} proved that the bounded geodesic length problem for $\mathbb Z_2\wr \mathbb Z^2$ is NP-complete. So if such an algorithm could be constructed to run in polynomial time, we would have P=NP.

A second example is the class of free metabelian groups --- Svetla Vassileva has shown they have normal forms (and hence word problem) computable in logspace \cite{\VassThesis}, and Miasnikov {\em et al.} \cite{\MiasFreeSolv} proved the bounded geodesic length  problem for these groups is NP-complete.

\section{Closure properties}
\label{sec:closure}

In this section we show that under certain conditions  $\mathcal C$-graph automaticity is preserved under change of group generating set, direct and free product.
Recall that by  Lemma~\ref{lem:langsclosed} the following  classes are closed under intersection with regular languages, finite intersection, $\e$-free homomorphism, and   inverse homomorphism:
%\begin{itemize}\item regular languages;\item $\mathscr C_k$;  \item $\mathscr S_k$;\item poly-context free languages; \item context-sensitive languages.\end{itemize}
regular languages, $\mathscr C_k$, $\mathscr S_k$, poly-context free languages, context-sensitive languages.
Moreover these classes all contain the class of regular languages.

\begin{lem}[Change of generators]
Let $G$ be a group with two symmetric generating sets $X$ and $Y$, $\Lambda$ a finite alphabet, and let $\mathcal C$ be a class of formal languages that is closed under finite intersection and  inverse homomorphism, and contains the class of regular languages.
If  $(G,X,\Lambda)$ is $\mathcal C$-graph automatic,  then  $(G,Y,\Lambda)$ is $\mathcal C$-graph automatic.
\end{lem}
\begin{proof}
Since we can use the same language
 $L\subseteq \Lambda^*$ for $(G,Y,\Lambda)$ as for $(G,X,\Lambda)$, it suffices to show that each language $L_y$ lies in the class $\mathcal C$.
 
 Let $Y_1\subseteq Y$ be the set of generators that do not equal the identity in $G$.
 For each $y\in Y_1$, choose $u_y\in X^+$ such that $u_y=_Gy$.
 Fix $y\in Y_1$ and suppose $u_y= x_1\dots x_k$ with $x_i\in X$.
 Consider convolutions of $k+1$ strings $v_i\in L$
\[\otimes(v_0,v_1,v_2,\dots, v_{k})\] so that $\overline{v_{i}}=_G \overline{v_{i-1}}x_i$ for $1\leq i\leq k$. Let $P$ be the language of all such convolutions.

 For each $x_i$ appearing in $u_y$ define a language $A_i$ of convolutions of $k+1$ strings over $\Lambda$ where rows
 $i$ and $i+1$ correspond to the language $L_{x_i}$, and all other rows can be any words in $\Lambda^*$.
Then $A_i$ is the inverse image of $L_{x_i}$ under the homomorphism which sends $\otimes(v_0,\dots, v_k)$ to $\otimes(v_{i-1},v_{i})$.

 Then $\bigcap_{i=1}^k A_i$ is  in $\mathcal C$ since the class is  closed under finite intersection.

Finally consider the $\e$-free homomorphism from $\bigcap_{i=1}^k A_i$ to $\otimes(L,L)$ defined by \[\otimes(v_0,v_1,v_2,\dots, v_{k})\mapsto \otimes(v_0,v_k).\]
Since $y$ is assumed to be non-trivial, the image this map is guaranteed to be $\e$-free.
The language $L_y$ is the image of $\bigcap_{i=1}^k A_i$ under this homomorphism, so is in $\mathcal C$.

To complete the proof, we must consider  $y \in Y\setminus Y_1$, that is, $y$ equals the identity element.
%the case that $Y$ contains generators $y$ that equal the identity element. 
In this case $L_y=\{\otimes(u,u) \mid u\in L\}$ which is regular, and so by assumption in $\mathcal C$.
\end{proof}

Note that the lemma holds when one or both of $X$ and $Y$ are countably infinite, since for each $y\in Y$ the word $u_y$  is a finite string of letters in $X$.
%By Lemma~\ref{lem:langsclosed} the classes of
%%%%Examples of classes $\mathcal C$ closed under finite intersection and inverse homomorphism  include counter, poly-context free, poly-indexed, and  context-sensitive.

\begin{lem}[Direct product]
Let $G$ and $H$ be  groups with symmetric generating sets $X$ and $Y$ respectively, $\Lambda$ and $\Gamma$  finite alphabets, and let $\mathcal C$ be a class of formal languages that is closed under intersection with regular languages, finite intersection and  inverse homomorphism.
If  $(G,X,\Lambda)$ and $(H,Y,\Gamma)$ are $\mathcal C$-graph automatic,  then the group $G\times H$ is $\mathcal C$-graph automatic.
\end{lem}
\begin{proof}
Assume $\Lambda$ and $\Gamma$ are disjoint.
Let $L_G\subset\Lambda^*$ and $L_H\subset \Gamma^*$ denote the languages of normal forms for each group,
and  $Z=\{(x,1_H), (1_G,y)\mid x\in X, y\in Y\}$  a generating set for $G\times H$.
Define a normal form $L=\otimes(L_G,L_H)$ for $G\times H$.

The language  $\otimes(L_G,\Gamma^*)$  is the inverse image of the homomorphism from $\otimes(L_G,\Gamma^*)$ to $L_G$ which sends $\otimes(u,v)$ to $u$, and similarly for
$\otimes(\Lambda^*, L_H)$. Then $L$ is the intersection of these languages and hence lies in the class $\mathcal C$.

For each $x\in X$ let $L_x$ be the multiplier language for the $\mathcal C$-graph automatic structure on $G$.
Define  \[L_1=\{\otimes(\otimes(u,w),\otimes(v,w))\mid u,v\in \Lambda^*, w\in \Gamma^*\},\]
  \[L_2=\{\otimes(\otimes(u,w),\otimes(v,z))\mid u,v\in \Lambda^*,  w\in L_H, z\in \Gamma^*\},\]  and
\[L_3=\{\otimes(\otimes(u,w),\otimes(v,z))\mid  u,v\in L_G, \ov =_G \ou x, w,z\in \Gamma^*\}.\]
Then $L_1$ is regular, $L_2$ is the inverse image of the homomorphism \[\phi:\otimes(\otimes(\Lambda^*,\Gamma^*),\otimes(\Lambda^*,\Gamma^*)) \rightarrow L_H\] given by  $\otimes(\otimes(a,b),\otimes(c,d))=b$, and $L_3$ is the inverse image of the homomorphism \[\phi:\otimes(\otimes(\Lambda^*,\Gamma^*),\otimes(\Lambda^*,\Gamma^*)) \rightarrow L_x\] given by  $\otimes(\otimes(a,b),\otimes(c,d))=\otimes(a,c)$, so $L_2$ and $L_3$ lie in $\mathcal C$.

It follows that  \[L_{(x,1_H)}=\{\otimes(\otimes(u,w),\otimes(v,w))\mid u,v\in L_G, \ov =_G \ou x, w\in L_H\}\]
is in $\mathcal C$ since it is the intersection $L_1\cap L_2\cap L_3$.

A similar argument applies to multiplier languages $L_{(1_G,y)}$.
\end{proof}
%%%%%NOTE KKM do semidirect then this is a cor.  If easy let's do it else not here.

For certain language classes $\mathcal C$ we prove that $\mathcal C$-graph automatic groups are closed under free product. The following argument is specific to the class 
of non-blind counter languages, and can be modified to apply to  poly-context-free, and context-sensitive languages. %, but not to blind counter languages.

\begin{lem}[Free product]\label{lem:freeproduct}
Let $G$ and $H$ be  groups with symmetric generating sets $X$ and $Y$ respectively, and $\Lambda$ and $\Gamma$  finite alphabets. 
If  $(G,X,\Lambda)$ is $\mathscr S_k$-graph automatic and $(H,Y,\Gamma)$ is $\mathscr S_l$-graph automatic, then $G\ast H$ is $\mathscr S_{\max\{k,l\}}$-graph automatic.
\end{lem}
\begin{proof}
Assume that $\Lambda$ and $\Gamma$ are distinct sets of symbols, and let $L_G\subset \Lambda^*,L_H\subset\Gamma^*$ be the normal form languages for $G,H$ respectively, and
$\lambda_0\in L_G$ and $\gamma_0\in L_H$ the normal form words for the identity in each language.

Define $L_1=L_G\setminus \{\lambda_0\}$; this is a $k$-counter language as it is the intersection of $L_G$ with the regular language $\Lambda^*\setminus \{\lambda_0\}$, and similarly $L_2=L_H\setminus \{\gamma_0\}$ is an $l$-counter language.
If $L_1$ contains the empty string, choose $u\in \Lambda^*\setminus L_G$ and replace $L_1$ by its image under the homomorphism from $L_1$ to $\Lambda^*$ which sends $\e$ to $u$ and  is the identity on all other strings. Then $L_1$ remains a $k$-counter language. Similarly if $L_2$ contains the empty string, it can be replaced.
Define
\[L=\left\{
\begin{array}{l|l}
\begin{array}{l}
\e, \\
\#u_1\#v_1\#\dots \#u_s\#v_s, \\
\#u_1\#v_1\#\dots \#v_{s-1}\#u_s, \\
\#v_1\#u_2\#\dots \#u_s\#v_s,  \\
\#v_1\#u_2\#\dots  \#v_{s-1}\#u_s\end{array} &  s>0, u_i\in L_1, v_i\in L_2\end{array} \right\}\]
over the alphabet $\{\#\}\cup \Lambda\cup \Gamma$.
There is an obvious bijection from $L$ to the free product, namely the map that deletes all $\#$,  sends $u_i$ to $\overline{u_i}$ and $v_i$ to $\overline{v_i}$. 

Let $M_1$ be the $k$-counter automaton accepting $L_1$, with start state $\tau_1$; analogously  let $M_2$ be the $l$-counter automaton  with start state $\tau_2$ accepting $L_2$.   Assume the sets of states of $M_1$ and $M_2$ are distinct. Define a nondeterministic, non-blind $\max\{k,l\}$-counter automaton $M$  as follows.
The states of $M$ are the states of $M_1$ and $M_2$ together with three new states $\kappa_0,\kappa_1,\kappa_2$. The start state for $M$ is $\kappa_0$, and accepting states are $\kappa_1$ and $\kappa_2$.
The edges in $M$ are as follows:
\begin{enumerate}
\item Every edge in $M_1$ is again an edge in $M$, where the first $k$ counters correspond to the $k$ counters in $M_1$.
\item Every edge in $M_2$ is again an edge in $M$, where the first $l$ counters correspond to the $l$ counters in $M_2$.
\item For each accept state $\tau_a$ in $M_1$, put an edge from $\tau_a$ to $\kappa_2$ labeled $\e_{=,\dots, =}$. Note that this transition is allowed only when all counters are zero.
\item For each accept state $\tau_a'$ in $M_2$,  put an edge from $\tau_a'$ to $\kappa_1$ labeled $\e_{=,\dots, =}$. Again, this edge is followed only when all counters are zero.
\item Put an edge labeled $\e$ from
$\kappa_0$ to $\kappa_1$, and an edge labeled $\e$ from
$\kappa_0$ to $\kappa_2$.
\item Put an edge labeled $\#$ from
$\kappa_1$ to $\tau_1$, and an edge labeled $\#$ from
$\kappa_2$ to $\tau_2$.
\end{enumerate}
See Figure~\ref{fig:Mfreeprod}.
Then $M$ is a non-blind non-deterministic $\max\{k,l\}$-counter automaton which accepts the language $L$.

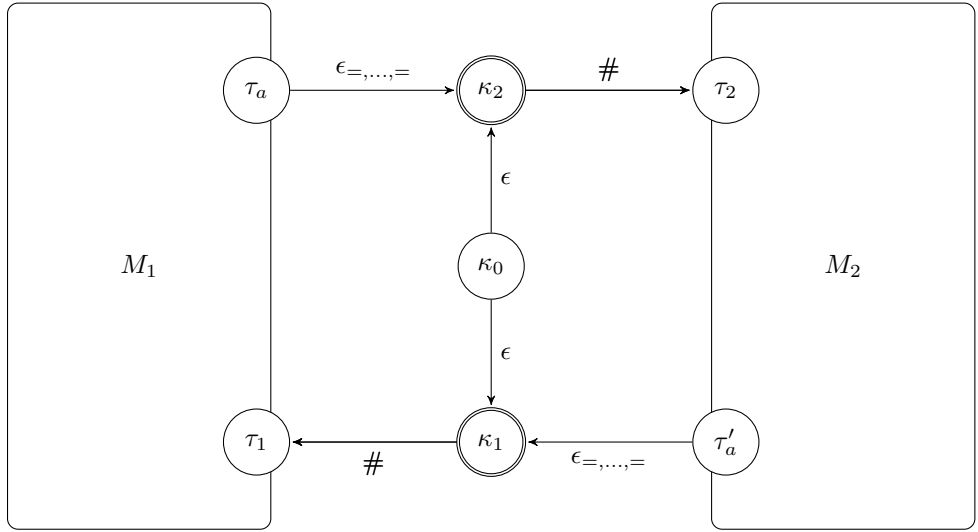
\begin{figure}[h]
\begin{center}
\begin{tikzpicture}[scale=.6, ->,>=stealth',shorten >=1pt,auto,node distance=2.5cm,  scale=1.3]
\tikzstyle{every state}=[fill=white,draw=black,text=black]

\tikzset{every loop/.style={min distance=10mm,in=60,out=120,looseness=10}}

  \node  [state, rectangle, rounded corners, minimum width=3.5cm, minimum height=7cm] (p) at (-3,3) {$M_1$};   %draw= gray,
    \node  [state, rectangle, rounded corners, minimum width=3.5cm, minimum height=7cm] (q) at (9,3) {$M_2$};

  \node  [state] (a) at (3,3) {$\kappa_0$};
   \node  [state, accepting] (b) at (3,0) {$\kappa_1$};
   \node  [state, accepting] (c) at (3,6) {$\kappa_2$};

  \node  [state] (d) at (-1,0) {$\tau_1$};
  \node  [state] (e) at (-1,6) {$\tau_a$};
  \node  [state] (f) at (7,0) {$\tau_a'$};
  \node  [state] (g) at (7,6) {$\tau_2$};

   \path (a) edge [right] node {$\e$} (b);
   \path (a) edge [right] node {$\e$} (c);

   \path (b) edge [below] node {$\#$} (d);
   \path (c) edge [above] node {$\#$} (g);

   \path (b) edge [below] node {$\#$} (d);
   \path (c) edge [above] node {$\#$} (g);
      
   \path (e) edge [above] node {$\e_{=,\dots,=}$} (c);
   \path (f) edge [below] node {$\e_{=,\dots,=}$} (b);

      \end{tikzpicture}
 \caption{Construction of the automaton $M$ in the proof of Lemma~\ref{lem:freeproduct}. Start state is $\kappa_0$ and accept states are $\kappa_1$ and $\kappa_2$.}

   \label{fig:Mfreeprod}
\end{center}\end{figure}

Let $x\in X$, and let $L_{G,x}$ be the multiplier language for the counter-graph automatic structure on $G$.  Analogously, for $y \in Y$ let $L_{H,y}$ be the multiplier language for the counter-graph automatic structure on $H$.

We will describe the multiplier language in the case of multiplication by $x\in X$ and leave the analogous case of $y\in Y$ to the reader.

The multiplier language $L_x=\{\otimes(p,q) \mid p,q\in L, \overline{q}=_{G\ast H} \overline{p}x\}\subseteq \otimes(L,L)$ for $G\ast H$ is accepted by a modified version of $M$ which we denote $M_x$, constructed as follows.
\begin{enumerate}
\item
Let $M_x$ initially have the same states  and transitions as $M$, with none labeled as accept states. Replace each edge label $\alpha \neq \e$ %\in \{\#\}\cup \Lambda\cup \Gamma$
by $\left( \begin{array}{c} \alpha \\ \alpha \end{array} \right)$.
\item
Let $\lambda_1 \lambda_2 \cdots \lambda_s\in L_H$ be the normal form word for $x$.  Add a new  state $\chi_1$ and a path from $\kappa_1$ to $\chi_1$
labeled ${\diamond \choose \#}{\diamond \choose \lambda_1}\dots {\diamond \choose \lambda_s}$. Declare $\chi_1$ to be an accept state.
This ensures that if $p$ is empty, or $p$ ends with a subword from $H$, that $\otimes(p,q)$ is accepted, where $\overline{q}=_G \overline{p}x$.   %%% (and so the path labeled $\otimes(p,p)$ in $M_x$ ends in $\kappa_1$)
\item
From $\kappa_1$ add an edge to a copy of the machine $L_{G,x}$ labeled $\e$.  Declare all previous accept states of this machine to be accept states of $M_x$.  If $p$ ends with a subword from $G$, say $p=\beta\gamma$ where $\gamma$ is the maximal suffix from $G$, then $\beta$ corresponds to a path through $M$ with an epsilon edge leading to $\tau_1$.  At that point, $L_{G,x}$ checks that the two suffix strings differ by $x$ in $G$.
%\item add analogous case for x in gen set for H.
\end{enumerate}
\end{proof}

\section{Examples}\label{sec:egs}

\subsection{Infinitely generated groups}
 \label{subsec:Finf}

The purpose of this example is to show that non-finitely generated %(but countable)
 groups are captured by the class of ${\mathcal C}$-graph automatic groups for appropriate ${\mathcal C}$. %(can we do it for graph automatic????)

\begin{prop} \label{prop:Finfty}
The free group
  $F_{\infty}=\langle  x_1, x_2, x_3, \dots \mid - \rangle$ on the countable set of generators $Y=\{x_i  \mid  i\in \mathbb Z_+\}$ is deterministic non-blind 2-counter-graph automatic. 
  \end{prop}
  
  \begin{proof}
  The idea is to represent generators and their inverses  as positive or negative unary integers.
Let $X=Y\cup Y^{-1}$, $\Lambda=\{p,n, 1\}$, and define a  homomorphism $\phi:X^*\rightarrow \Lambda^*$
by $\phi(x_i)=p 1^i$ and $\phi(x_i^{-1})=n 1^i$.
For example, $x_2^3x_5^{-1}$ is mapped to $p11p11p11n11111$.
The set of freely reduced finite strings of generators is  a normal form for $F_{\infty}$, so define a normal form  $L\subseteq \Lambda^*$ to be the image of this set under $\phi$. Note that the identity corresponds to the empty string $\e$.

Let $L_1\subseteq \Lambda^*$ be the set of strings of the form $r_11^{\eta_1}\dots r_k1^{\eta_k}$ where $r_i\in\{p,n\}$ and $\eta_i\in \mathbb Z^+$.
Let $L_{2}$ be the set of strings in  $L_1$  where  $r_{2i-1}\neq r_{2i}$ implies
$\eta_{2i-1} \neq  \eta_{2i}$, and $L_{3}$ the strings in $L_1$ where $r_{2i}\neq r_{2i+1}$ implies  $\eta_{2i}\neq \eta_{2i+1}$, for $i\geq 1$.
That is, in $L_{2}$ substrings $r_{2i-1}1^{\eta_{2i-1}}r_{2i}1^{\eta_{2i}}$ represent a freely reduced pair, and  in $L_{3}$ substrings $r_{2i}1^{\eta_{2i}}r_{2i+1}1^{\eta_{2i+1}}$ represent a freely reduced pair.   For example, $n1p11n11p1$ is in $L_{2}$ but not $L_{3}$.  The intersection $L_{2}\cap L_{3}$ is then the normal form language $L$. %, and is 2-counter by Lemma~\ref{}.

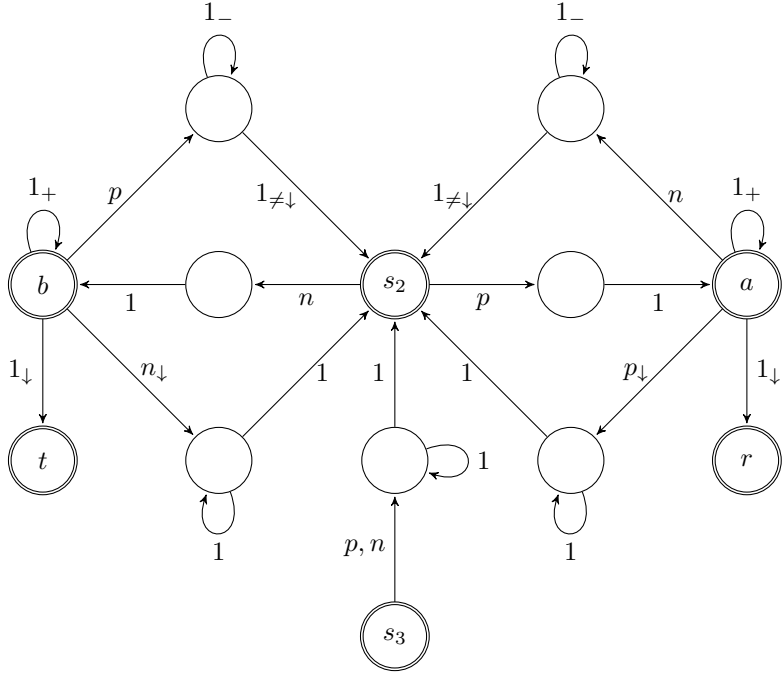
\begin{figure}[h]
\begin{center}
\begin{tikzpicture}[scale=.6, ->,>=stealth',shorten >=1pt,auto,node distance=2.5cm,  scale=1.3]
\tikzstyle{every state}=[fill=white,draw=black,text=black]

%%\tikzset{every loop/.style={min distance=10mm,in=60,out=120,looseness=10}}

\tikzset{every loop/.style={min distance=10mm,in=70,out=110,looseness=5}}

  \node  [state, accepting] (s) at (6,3) {$s_2$};
   \node  [state, accepting] (a) at (12,3) {$a$};
   \node  [state, accepting] (b) at (0,3) {$b$};

  \node  [state] (c) at (9,6) {};
  \node  [state] (d) at (9,3) {};
  \node  [state] (e) at (9,0) {};
  \node  [state,accepting] (C) at (12,0) {$r$};
  \node  [state,accepting] (D) at (0,0) {$t$};

  \node  [state] (x) at (3,6) {};
  \node  [state] (y) at (3,3) {};
  \node  [state] (z) at (3,0) {};

   \path (s) edge [below] node {$p$} (d);
   \path (d) edge [below] node {$1$} (a);
   \path (a) edge [right] node {$n$} (c);
   \path (a) edge [left] node {$p_{\downarrow}$} (e);
   \path (c) edge [left] node {$1_{\neq\downarrow}$} (s);
   \path (e) edge [left] node {$1$} (s);

   \path (a) edge [right] node {$1_{\downarrow}$} (C);
   \path (b) edge [left] node {$1_{\downarrow}$} (D);

   \path (s) edge [below] node {$n$} (y);
   \path (y) edge [below] node {$1$} (b);
   \path (b) edge [left] node {$p$} (x);
   \path (b) edge [right] node {$n_{\downarrow}$} (z);
   \path (x) edge [left] node {$1_{\neq\downarrow}$} (s);
   \path (z) edge [right] node {$1$} (s);

%%North

      \path (c) edge [loop, above] node {$1_{-}$} (c);
   \path (x) edge [loop, above] node {$1_{-}$} (x);

 \path (b) edge [loop, above] node {$1_{+}$} (b);
 \path (a) edge [loop, above] node {$1_{+}$} (a);

%%South
\tikzset{every loop/.style={min distance=10mm,in=250,out=290,looseness=5}}

      \path (e) edge [loop, below] node {$1$} (e);
      \path (z) edge [loop, below] node {$1$} (z);

%%West
\tikzset{every loop/.style={min distance=10mm,in=160,out=200,looseness=5}}

%%East
\tikzset{every loop/.style={min distance=10mm,in=340,out=20,looseness=5}}

   \node  [state] (S) at (6,0) {};
   \node  [state, accepting] (SS) at (6,-3) {$s_3$};

   \path (S) edge [left] node {$1$} (s);
   \path (SS) edge [left] node {$p,n$} (S);

    \path (S) edge [loop, right] node {$1$} (S);

      \end{tikzpicture}
 \caption{Deterministic non-blind 1-counter automaton accepting the language $L_{3}$ in the proof of Proposition~\ref{prop:Finfty}.  The start state is $s_3$. Accept states are $s_2,s_3,a,b,r,t$.
 The automaton for $L_2$ is identical with  start state  $s_2$. }

   \label{fig:Finfty}
\end{center}\end{figure}

A deterministic non-blind 1-counter automaton accepting  $L_3$
 is shown in Figure~\ref{fig:Finfty}. The automaton accepting $L_{2}$ is obtained from this by setting $s_2$ to be the start state.
Recall that the notation $1_{\neq\downarrow}$ means if the counter is nonzero, read 1 and set the counter to 0.
%Note that the prefix of an accepted string is $p1^i$ or $n1^i$ for any $i\geq 1$, the nafter this the automaton checks for freely reduced pairs. Strings representing words in $X^*$ of odd length end in  $s_2$, and strings representing words in $X^*$ of   even length will end in states $r,t$ (or $a,b$ if the last letter is $x_1^{\pm 1}$, and $s_3$ if the empty string).

Then  $L=L_{2}\cap L_{3}$ is deterministic non-blind 2-counter by Lemma~\ref{lem:counterintersect}.

%\FloatBarrier

The multiplier language $L_{x_i}$ for the generator $x_i$
 is the set of strings in $\otimes (L,L)$ of the form
\[{r_1 \choose r_1}   {1\choose 1}^{\eta_1} {r_2 \choose r_2}   {1\choose 1}^{\eta_2}  \dots  {r_{k}\choose r_{k}}  {1\choose 1}^{\eta_{k}} {\diamond \choose p}{\diamond\choose 1}^i\]
if $r_k=p$ or  $\eta_k\neq i$,
and otherwise if $r_k=n$ and $\eta_k=i$
\[{r_1 \choose r_1}   {1\choose 1}^{\eta_1} {r_2 \choose r_2}   {1\choose 1}^{\eta_2}  \dots  {r_{k-1}\choose r_{k-1}}  {1\choose 1}^{\eta_{k-1}} {n\choose \diamond}{1\choose \diamond}^i.  \]

Define $L_{x_i}^+$ to be the  regular language is given by the regular expression
\[\left\{{1\choose 1}, {p\choose p}, {n\choose n}\right\}^*\left\{{\diamond \choose p}{\diamond \choose 1}^i\right\},\] and $L_{x_i}^-$ the language given by the regular expression \[\left\{{1\choose 1}, {p\choose p}, {n\choose n}\right\}^*\left\{{n \choose \diamond}{1\choose \diamond }^i\right\}.\]

Next consider the language  $ \otimes(\Lambda^*, L)$. Modify the automaton in Figure \ref{fig:Finfty} by replacing edges labeled $x_{\#}$ (where $x\in\{p,n,1\}$ and $\#$ denotes some counter instructions) by  four edges labeled ${p\choose x}_{\#}, {n\choose x}_{\#}, {1\choose x}_{\#}, {\diamond \choose x}_{\#}$.
The intersection of the two languages of strings accepted by this automaton with start state either $s_2$ or $s_3$ is
the language  $ \otimes(\Lambda^*, L)$, and is  deterministic non-blind 2-counter.

A similar argument shows that $ \otimes(L,\Lambda^*)$ is  deterministic non-blind 2-counter.  Then $L_{x_i}$ is the union of $L_{x_i}^+\cap \otimes(\Lambda^*, L)$  and  $L_{x_i}^-\cap  \otimes(L,\Lambda^*)$, and so is deterministic non-blind 2-counter.
\end{proof}

\subsection{Baumslag-Solitar groups}

In \cite{\KKM} the solvable Baumslag-Solitar groups are shown to be graph automatic. Here we show that the non-solvable Baumslag-Solitar groups are blind deterministic 3-counter-graph automatic.

\begin{prop}
Let $2\leq m<n$. Then $BS(m,n)=\langle a,t  \mid  ta^mt^{-1}=a^n\rangle$ is  blind deterministic 3-counter-graph automatic.
\end{prop}
\begin{proof}
Any word in $\{a^{\pm1}, b^{\pm1}\}^*$ can be transformed into a normal form  for
the corresponding group element by ``pushing'' each $a$ and $a^{-1}$ in the
word as far to the right as possible and freely reducing using the identities
\[\begin{array}{llllllll}
a^{\pm 1}a^{\mp 1}=1, & a^{\pm n} t = t a^{\pm m}, & a^{-i}t=a^{n-i}ta^{-m},\\
 t^{\pm 1}t^{\mp 1}=1, &
 a^{\pm m} t^{-1} = t^{-1} a^{\pm n},
&  a^{-j}t^{-1}= a^{m-j} t^{-1} a^{-n}.
\end{array}\]
where $0<i<n$ and $ 0<j<m$, so that only positive powers of $a$ appear before a
$t^{\pm 1}$ letter.
 The resulting word can be written as $P a^N$, where  $P$ is a freely reduced word in the alphabet
$\Pi=\{t, at, \dots a^{n-1}t, t^{-1}, at^{-1}, \dots a^{m-1}t^{-1}  \}$
(see for example  \cite{\LS}  p.181).  Let $\Gamma\subseteq \Pi^*$ be the set of freely reduced words in $ \Pi^*$. %Clearly $\Gamma$ is a regular language.

It is clear that the language of the words of the form $Pa^N$  with $P\in\Gamma, N\in\mathbb N$ is regular, and in bijection with the group. The idea for the counter-graph automatic structure is to represent the integer $N$ in two different ways, so that multiplication by the generator $t$ can be easily recognised.

For $N\in \mathbb Z$, if $N$ is positive write $N=pm+r=qn+s$ with $0\leq r<m$ and $0\leq s<n$; if $N$ is negative write $N=-(pm+r)=-(qn+s)$; and otherwise write $N=0$.
 Define $L$ to be the language %representing words  of the form $Pa^N$ with $N$ positive, negative and zero respectively by %%%>0, N<0, N=0$ respectively by
$$L =  \left\{ \begin{array}{l|l} P\#1^r\#1^p\#1^s\#1^q, &
P\in\Gamma,\\
     P\#(-1)^r\#(-1)^p\#(-1)^s\#(-1)^q, & r\in[0,m), s\in[0,n),\\  P\#\#\#\# &r+pm=s+qn,\\ & r+pm>0\end{array}\right\}.$$
Then $L$ is in bijection with words of the form $Pa^N$ for $N$ positive, negative and zero, so is a normal form for $BS(m,n)$ over the alphabet
$\Lambda =\Pi\cup\left\{1, -1, \# \right\}$.

For example, in $BS(4,7)$: \begin{itemize}
\item the string $at\#111\#1\#\#1$ represents the word $ata^{7}$;
\item  the string $at\#11111\#1\#\#1$ is rejected since $r=5$ is not less than $m=4$;
\item  the string $at\#11\#11\#1\#1$ is rejected since $r+pm=10$ whereas $s+qn=8$.
\end{itemize}

Let $L_1$ be the language \[
L_1=\left\{\begin{array}{l|l}P\#1^r\#1^{p}\#1^s\#1^{q},  &  P\in \{a,t^{\pm 1}\}^*,\\
P\#(-1)^r\#(-1)^{p}\#(-1)^s\#(-1)^{q},  & r,p,s,q\in\N,\\
P\#\#\#\#   &    r+pm=s+qn>0\end{array}\right\}.\]
Then $L_1$ is accepted by the  blind deterministic 1-counter automaton shown in Figure~\ref{fig:automBS}.

Let $L_2$ be the regular language of strings
 \[
L_2=\left\{\begin{array}{l|l}P\#1^r\#1^{p}\#1^s\#1^{q}, & P\in \Gamma,  \\
P\#(-1)^r\#(-1)^{p}\#(-1)^s\#(-1)^{q},  & r,s,p,q\in\mathbb N, \\
P\#\#\#\#  & r<m, s<n \end{array}\right\}.\]
Then $L=L_1\cap L_2$ is a blind-1-counter language.

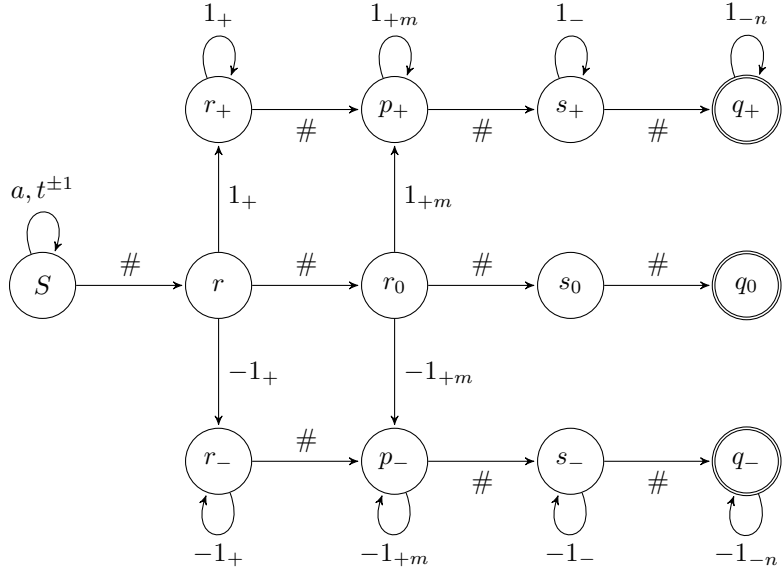
\begin{figure}[h]
\begin{center}
\begin{tikzpicture}[scale=.6, ->,>=stealth',shorten >=1pt,auto,node distance=2.5cm,  scale=1.3]
\tikzstyle{every state}=[fill=white,draw=black,text=black]

\tikzset{every loop/.style={min distance=10mm,in=70,out=110,looseness=5}}

   \node  [state] (SS) at (1,4) {$S$};

   \node  [state] (r) at (4,1) {$r_-$};
      \node  [state] (p) at (7,1) {$p_-$};
   \node  [state] (s) at (10,1) {$s_-$};
  \node  [state, accepting] (q) at (13,1) {$q_-$};

   \node  [state] (X) at (4,4) {$r$};

 \path (SS) edge [above] node {$\#$} (X);

   \path (X) edge [right] node {$-1_+$} (r);
      \path (r) edge [above] node {$\#$} (p);
         \path (p) edge [below] node {$\#$} (s);
            \path (s) edge [below] node {$\#$} (q);

   \node  [state] (R) at (4,7) {$r_+$};
      \node  [state] (P) at (7,7) {$p_+$};
   \node  [state] (S) at (10,7) {$s_+$};
  \node  [state, accepting] (Q) at (13,7) {$q_+$};

    \path (X) edge [right] node {$1_+$} (R);
      \path (R) edge [below] node {$\#$} (P);
         \path (P) edge [below] node {$\#$} (S);
            \path (S) edge [below] node {$\#$} (Q);

      \path (SS) edge [loop, above] node {$a, t^{\pm 1}$} (SS);

      \path (R) edge [loop, above] node {$1_+$} (R);
      \path (P) edge [loop, above] node {$1_{+m}$} (P);
      \path (S) edge [loop, above] node {$1_-$} (S);
      \path (Q) edge [loop, above] node {$1_{-n}$} (Q);

%%South
\tikzset{every loop/.style={min distance=10mm,in=250,out=290,looseness=5}}

      \path (r) edge [loop] node {$-1_+$} (r);
      \path (p) edge [loop] node {$-1_{+m}$} (p);
      \path (s) edge [loop] node {$-1_-$} (s);
      \path (q) edge [loop] node {$-1_{-n}$} (q);

   \node  [state] (Y) at (7,4) {$r_0$};
    \path (X) edge [above] node {$\#$} (Y);

   \node  [state] (YY) at (10,4) {$s_0$};   %%$\begin{array}{l}r=0\\p=0\end{array}$};
    \path (Y) edge [above] node {$\#$} (YY);
   \node  [state, accepting] (YYY) at (13,4) {$q_0$};
    \path (YY) edge [above] node {$\#$} (YYY);

 \path (Y) edge [right] node {$1_{+m}$} (P);
 \path (Y) edge [right] node {$-1_{+m}$} (p);

      \end{tikzpicture}
 \caption{1-counter automaton accepting the language $L_1$ for $BS(m,n)$.  Accept states are $q_+,q_0$ and $q_-$.  %The notation $1_{+m}$ means add $m$ to the counter on reading input letter $1$.
 The counter checks the equation $r+pm=s+qn$ is satisfied.}

   \label{fig:automBS}
\end{center}
\end{figure}

%\FloatBarrier

Now we turn to the multiplier languages $L_a$ and $L_t$.  %We use lower-case greek letters to denote (arbitrary) nonnegative integers.

First observe that the languages $\otimes(L,\Lambda^*)$ and $\otimes(\Lambda^*,L)$ are blind-1-counter, and so $\otimes(L,L)=\otimes(L,\Lambda^*)\cap \otimes(\Lambda^*,L)$ is a blind-2-counter language  by Lemma~\ref{lem:counterintersect}.

We will describe $L_a$ as the union of a set of languages intersected with  $\otimes(L,L)$. Note that $L_a$ is the set of strings $\otimes(u,v)$ where $\overline u=Pa^N, \overline v=Pa^{N+1}$.
Recall that the regular expression $\{1\}^*\#\{1\}^*$ denotes  the set of  strings in $\{1,\#\}^*$ with exactly one $\#$ letter.
 The languages are as follows,  for $0\leq r\leq m-2$:
\begin{itemize}
\item
$ \mathcal L_r=\left\{\begin{array}{l|l}
\otimes\left(\begin{array}{l}
P\#1^r\#1^p\#Q,  \\
P\#1^{r+1}\#1^p\#R\end{array}\right)

& \begin{array}{l} P\in \{a,t^{\pm 1}\}^*, \\
 p\in \mathbb N, \\
 Q,R\in  \{1\}^*\#\{1\}^* \end{array} \end{array}\right\};$
 \medskip

\item
$ \mathcal L_{m-1}=\left\{\begin{array}{l|l}
\otimes\left(\begin{array}{l}P\#1^{m-1}\#1^p\#Q,  \\ P\#\#1^{p+1}\#R\end{array}\right)
& \begin{array}{l} P\in \{a,t^{\pm 1}\}^*, \\
 p\in \mathbb N,\\
 Q,R\in \{1\}^*\#\{1\}^*  \end{array} \end{array}\right\};$
  \medskip

\item
$ \mathcal K_{r+1}=\left\{\begin{array}{l|l}
\otimes\left(\begin{array}{l}P\#(-1)^{r+1}\#(-1)^p\#Q,  \\ P\#(-1)^r\#(-1)^p\#R\end{array}\right)
& \begin{array}{l} P\in \{a,t^{\pm 1}\}^*, \\
 p\in \mathbb N,\\
 Q,R\in \{-1\}^*\#\{-1\}^*  \end{array} \end{array}\right\};$

 \medskip

\item
$ \mathcal K_0=\left\{\begin{array}{l|l}
\otimes\left(\begin{array}{l}P\#\#(-1)^{p+1}\#Q,  \\  P\#(-1)^{m-1}\#R\end{array}\right)
& \begin{array}{l} P\in \{a,t^{\pm 1}\}^*, \\
 p\in \mathbb N,\\
 Q,R\in  \{-1\}^*\#\{-1\}^* \end{array} \end{array}\right\}.$

\end{itemize}

%Observe that since we intersect with  $\otimes(L,L)$, the values taken by lower-case greek letters do not matter.

These languages are designed simply to check the condition that  $\overline u= Pa^N, \overline v=Pa^{N+1}$.
Each language is regular, so its intersection with $\otimes(L,L)$ is a
blind 2-counter language. It follows that $L_a$ is blind 2-counter.

%It s clear that each language is blind 1-counter, and further one can show that each is in fact regular.

\medskip

Now we come to the language $L_t$. %The reason for writing the suffix $a^N$ of a normal form word in two ways will now become apparent.
We will again intersect  with the blind 2-counter language $\otimes(L,L)$.
We must accept strings $\otimes(u,v)$ for words $u,v\in L$ with $\overline u=Pa^N$ and $\overline v=Pa^Nt$. We consider the following cases, which depend on  whether
or not $P$ ends in $t^{-1}$,  and  whether or not $n$ divides $N$.

\medskip

\noindent\textbf{Case 1}  $P$ ends in $t$ or is empty:

 For $N\geq 0$, write  $N=qn+s$ with $0\leq s<n$.  Then $a^Nt=a^sta^{qm}$. This gives  strings of the form
  \[\otimes(P\#1^{\alpha}\#1^{\beta}\#1^s\#1^q,Pa^st\#\#1^q\#1^{\gamma}\#1^{\delta})\]
where $\alpha,\beta,\gamma, \delta$ are the appropriate integers.  Note that  there is no cancelation between $P$ and the letters added, since $P$ is either empty or ends in $t$.

For $N<0$, write $N=-(qn+s)$ with $0\leq s<n$. Then   $a^Nt=a^{-s}ta^{-qm}$. If $s=0$ then this gives the set of strings
 \[\otimes(P\#(-1)^{\alpha}\#(-1)^{\beta}\#\#(-1)^q,Pt\#\#(-1)^q\#(-1)^{\gamma}\#(-1)^{\delta}).\]
If $s>0$ then $a^Nt=a^{-s}ta^{-qm}=a^{n-s}ta^{-m-qm}$ which gives the
set of strings
\[\otimes(P\#(-1)^{\alpha}\#(-1)^{\beta}\#(-1)^s\#(-1)^q,Pa^{n-s}t\#\#(-1)^{q+1}\#(-1)^{\gamma}\#(-1)^{\delta}).\]
Again there is no cancelation between $P$ and the letters added, since $P$ is either empty or ends in $t$.

These strings can be obtained by intersecting the following languages with $\otimes(L,L)$:
\begin{itemize}
\item
$ \mathcal U_s=\left\{\begin{array}{l|l}
\otimes\left(\begin{array}{l}P\#Q\#1^s\#1^q, \\
                                                 Pa^st\#\#1^q\#R\end{array}\right)

& \begin{array}{l} P\in \{\e, wt \ : \ w\in \{a,t^{\pm 1}\}^*\}, \\
 q\in \mathbb N, \\
 Q,R\in  \{1\}^*\#\{1\}^* \end{array} \end{array}\right\}$\\
 for $0\leq s\leq n-1$,

 \medskip

\item  $ \mathcal V_s=\left\{\begin{array}{l|l}
\otimes\left(\begin{array}{l}P\#Q\#(-1)^s\#(-1)^q, \\
                                                 Pa^{n-s}t\#\#(-1)^{q+1}\#R\end{array}\right)

& \begin{array}{l} P\in \{\e, wt \ : \ w\in \{a,t^{\pm 1}\}^*\}, \\
 q\in \mathbb N, \\
 Q,R\in  \{1\}^*\#\{1\}^* \end{array} \end{array}\right\}$\\
 for $1\leq s\leq n-1$.

 \medskip

\item  $ \mathcal V_0=\left\{\begin{array}{l|l}
\otimes\left(\begin{array}{l}P\#Q\#\#(-1)^q, \\
                                                 Pt\#\#(-1)^{q}\#R\end{array}\right)

& \begin{array}{l} P\in \{\e, wt \ : \ w\in \{a,t^{\pm 1}\}^*\}, \\
 q\in \mathbb N, \\
 Q,R\in  \{1\}^*\#\{1\}^* \end{array} \end{array}\right\}$.
\end{itemize}

The languages $\mathcal U_s, V_s$ for $0\leq s\leq n-1$ are blind 1-counter --- the counter is used to  check the entries $(\pm 1)^q$ are the same in each component of the convoluted string.

\medskip

\noindent\textbf{Case 2}  $P$ ends in $t^{-1}$, and $n$ does not divide $N$.

In this case $N=qn+s$ with $0<|s|<n$.

 For $N\geq 0$ write $N=qn+s$ with $s>0$.
 Then $Pa^Nt=Pa^sta^{qm}$ where $Pa^st$ has no cancelation so is in normal form. This gives the set of strings
 \[\otimes(P\#1^{\alpha}\#1^{\beta}\#1^s\#1^q,Pa^st\#\#1^q\#1^{\gamma}\#1^{\delta}).\]

For $N<0$, write $N=-(qn+s)$ with $s>0$.
Then   \[a^Nt=a^{-s}ta^{-qm}=a^{n-s}ta^{-m-qm}\]
and so $Pa^Nt=Pa^{n-s}ta^{-m-qm}$ and $P$ does not cancel, so this gives
 the
set of strings
\[\otimes(P\#(-1)^{\alpha}\#(-1)^{\beta}\#(-1)^s\#(-1)^q,Pa^{n-s}t\#\#(-1)^{q+1}\#(-1)^{\gamma}\#(-1)^{\delta}).\]

These strings can be obtained by intersecting the following languages with $\otimes(L,L)$:
\begin{itemize}
\item
$ \mathcal W_s=\left\{\begin{array}{l|l}
\otimes\left(\begin{array}{l}P\#Q\#1^s\#1^q, \\
                                                 Pa^st\#\#1^q\#R\end{array}\right)

& \begin{array}{l} P\in \{\e, wt^{-1} \ : \ w\in \{a,t^{\pm 1}\}^*\}, \\
 q\in \mathbb N, \\
 Q,R\in  \{1\}^*\#\{1\}^* \end{array} \end{array}\right\}$\\
 for $1\leq s\leq n-1$,

 \medskip

\item  $ \mathcal X_s=\left\{\begin{array}{l|l}
\otimes\left(\begin{array}{l}P\#Q\#(-1)^s\#(-1)^q, \\
                                                 Pa^{n-s}t\#\#(-1)^{q+1}\#R\end{array}\right)

& \begin{array}{l} P\in \{\e, wt^{-1} \ : \ w\in \{a,t^{\pm 1}\}^*\}, \\
 q\in \mathbb N, \\
 Q,R\in  \{-1\}^*\#\{-1\}^* \end{array} \end{array}\right\}$\\
 for $1\leq s\leq n-1$.

\end{itemize}

Again the languages $\mathcal W_s, X_s$ for $1\leq s\leq n-1$ are blind 1-counter --- the counter is used to  check the entries $(\pm 1)^q$ are the same in each component of the convoluted string.

\medskip

\noindent\textbf{Case 3}  $P$ ends in $t^{-1}$, and $n$  divides $N$.

Put $P=Ta^ct^{-1}$, where $c\in[0,m)$ and $T$ is empty or ends in $t^{\pm 1}$.
Since we will intersect with $\otimes(L,L)$ we don't care whether $Ta^ct^{-1}$ is freely reduced or not.

 For $N\geq 0$ write $N=qn$ so  \[Pa^Nt=Pta^{qm}=Ta^ct^{-1}ta^{qm}=Ta^{c+qm}.\]
 This gives the set of strings
 \[\otimes(Ta^ct^{-1}\#1^{\alpha}\#1^{\beta}\#\#1^q,T\#1^c\#1^q\#1^{\gamma}\#1^{\delta}).\]

For  $N<0$ write $N=-(qn)$  and
 \[Pa^Nt=Pta^{-qm}=Ta^ct^{-1}ta^{-qm}=Ta^{c-qm}=Ta^{c-m}a^{-(q-1)m}\]
 This gives the set of strings
  \[\otimes(Ta^ct^{-1}\#(-1)^{\alpha}\#(-1)^{\beta}\#\#(-1)^q,T\#(-1)^{m-c}\#(-1)^{q-1}\#(-1)^{\gamma}\#(-1)^{\delta}).\]

These strings can be obtained by intersecting the following languages with $\otimes(L,L)$:
\begin{itemize}
\item
$ \mathcal Y_c=\left\{\begin{array}{l|l}
\otimes\left(\begin{array}{l}Ta^ct^{-1}\#Q\#\#1^q, \\
                                                 T\#1^c\#1^q\#R\end{array}\right)

& \begin{array}{l} T\in \{a,t^{\pm 1}\}^*, \\
 q\in \mathbb N, \\
 Q,R\in  \{1\}^*\#\{1\}^* \end{array} \end{array}\right\}$\\
 for $0\leq c\leq n-1$,

 \medskip

\item  $ \mathcal Z_c=\left\{\begin{array}{l|l}
\otimes\left(\begin{array}{l}Ta^ct^{\pm 1}\#Q\#\#(-1)^q, \\
                                                 T\#(-1)^c\#(-1)^q\#R\end{array}\right)

& \begin{array}{l} T\in \{a,t^{\pm 1}\}^*, \\
c\in[0,n),\\
 q\in \mathbb N, \\
 Q,R\in  \{-1\}^*\#\{-1\}^* \end{array} \end{array}\right\}$\\
 for $0\leq c\leq n-1$,

\end{itemize}

Once again the languages $\mathcal Y_c, \mathcal Z_c$ for $0\leq c\leq n-1$  are blind 1-counter --- the counter is used to  check the entries $(\pm 1)^q$ are the same in each component of the convoluted string.

It follows that the language $L_t$ is the union of the languages $\mathcal U_i,\mathcal V_i,\mathcal W_i,\mathcal X_i,\mathcal Y_i,\mathcal Z_i$ each intersected with $\otimes(L,L)$ and is therefore blind deterministic 3-counter.
\end{proof}

We remark that the above normal form language is not quasigeodesic. In \cite{BurilloE}  Burillo and the first author find a
 metric estimate for  $BS(m,n)$. It is shown that the geodesic length of the element equal to $a^N$ is $O(\log N)$, while the normal form representative given above has length $O(N/m+N/n)=O(N)$.

%\section*{References}

\bibliography{refs} \bibliographystyle{plain}

\end{document}